%% file: main.tex
\documentclass[preprint,11pt,authoryear]{elsarticle}

\usepackage{setspace}
\usepackage{amsmath,amssymb,mathtools}
\usepackage{booktabs,array,multirow,tabularx}
\usepackage{graphicx}
\usepackage{placeins}
\usepackage{url}
\usepackage[hidelinks]{hyperref}
\usepackage{caption}

\usepackage{enumitem}
\usepackage{xcolor}
\usepackage{microtype}

\usepackage{float}
\usepackage{arydshln}

\allowdisplaybreaks[4]
\usepackage{algorithm}
\usepackage{algorithmic}
\usepackage{amsthm}
\newtheorem{proposition}{Proposition}
\usepackage[margin=1in]{geometry}

\makeatletter
\newenvironment{breakablealgorithm}
{%
    \par\addvspace{\intextsep}
    \refstepcounter{algorithm}%
    \hrule height .8pt depth 0pt
    \kern 2pt
    \renewcommand{\caption}[2][\relax]{%
        {\raggedright
        \textbf{\ALG@name~\thealgorithm:} ##2\par}%
        \ifx\relax##1\relax
            \addcontentsline{loa}{algorithm}
            {\protect\numberline{\thealgorithm}##2}%
        \else
            \addcontentsline{loa}{algorithm}
            {\protect\numberline{\thealgorithm}##1}%
        \fi
        \kern 2pt
        \hrule
        \kern 2pt
    }%
}
{%
    \kern 2pt
    \hrule
    \par\addvspace{\intextsep}
}
\makeatother

\newcommand{\Smain}{\ensuremath{\mathcal{S}^{\mathrm{main}}}}

\newcommand{\Spercentile}{\ensuremath{\mathcal{S}^{\mathrm{perc}}}}

\begin{document}

\begin{frontmatter}

\title{Multi-Objective Planning for Healthcare Facility Resilience: Mitigate Now or Respond Later?}

\author[ut]{Gizem Toplu-Tutay\corref{cor1}\fnref{present}}
\ead{gizem@utexas.edu}
\author[ut]{John J. Hasenbein}
\author[bbr]{Matt Kammer-Kerwick}
\author[ut]{Erhan Kutanoglu}

\cortext[cor1]{Corresponding author.}
\fntext[present]{Present address: The Earth Commons Institute,
Georgetown University, Washington, DC, United States.}
\affiliation[ut]{organization={Operations Research and Industrial Engineering, The University of Texas at Austin},
            addressline={204 E. Dean Keeton Street},
            city={Austin},
            state={Texas},
            postcode={78712},
            country={United States}}
            
\affiliation[bbr]{
    organization={Bureau of Business Research, IC² Institute,
    The University of Texas at Austin},
    addressline={2815 San Gabriel Street},
    city={Austin},
    state={Texas},
    postcode={78705},
    country={United States}
}

\begin{abstract}
Flooding can damage healthcare facilities, interrupt local care capacity, and trigger costly patient evacuations. Addressing these risks requires long-term resilience investments under uncertainty. We develop a bi-objective two-stage stochastic optimization model that jointly determines permanent facility hardening and scenario-dependent evacuation decisions. The first objective minimizes expected evacuation, physical damage, and business-interruption costs. The second minimizes a service-disruption impact index that combines the duration and scale of service loss with a place-based social-vulnerability weight. We develop an exact Benders decomposition and a scalable Lagrangian-dual method with tailored primal recovery, embedding both within an adaptive procedure for constructing informative Pareto frontiers. A case study of 3,752 hospitals and nursing homes in Texas evaluates the framework under climate-informed tropical-cyclone flood scenarios. The results show that minimizing economic losses alone can systematically allocate less protection to facilities located in socially vulnerable areas, while moderate movement along the Pareto frontier can substantially reduce disruption impacts at limited additional expected cost. The exact decomposition solves all tested instances, including larger stress tests for which the extensive formulation becomes memory-limited, while the Lagrangian method provides substantially faster high-quality solutions.
\end{abstract}

\begin{keyword}
OR in strategic planning \sep Resilience \sep Two-stage stochastic programming \sep Multi-objective optimization \sep Decomposition
\end{keyword}

\end{frontmatter}

\input{01_introduction}
\input{02_model}

\input{03_solution_approaches}
\input{04_case_study}
\input{05_results}

\input{06_conclusion}



\section*{Funding}
This work was supported by the Bureau of Business Research (BBR) at the IC$^2$ Institute, The University of Texas at Austin.

\section*{Data availability}
Data and code are archived on Zenodo at
\url{https://doi.org/10.5281/zenodo.21510695}.

\clearpage
\appendix

\renewcommand{\thesection}{S\arabic{section}}
\renewcommand{\thefigure}{S\arabic{figure}}
\renewcommand{\thetable}{S\arabic{table}}
\renewcommand{\theequation}{S\arabic{equation}}

\setcounter{section}{0}
\setcounter{figure}{0}
\setcounter{table}{0}
\setcounter{equation}{0}

\input{7_supp}

\bibliographystyle{elsarticle-harv}
\bibliography{references}

\end{document}

%% file: 01_introduction.tex
\section{Introduction}
\label{sec:introduction}

Public agencies responsible for disaster preparedness and mitigation must
allocate limited resilience funds before the location and severity of future
disruptions are known. Climate-related hazards increasingly threaten lifeline
systems, generating substantial economic losses and uneven burdens across
communities
\citep{noaa_national_centers_for_environmental_information_ncei,
hallegatte_poverty_2020}. Resilience planning therefore poses a strategic
capital-allocation problem: which assets should be protected, to what extent,
and under uncertain future hazard conditions?

This problem is particularly consequential for healthcare facilities. Flooding can
damage hospitals and nursing homes, interrupt local care capacity, and trigger
patient evacuations and prolonged business interruption
\citep{valente_impacts_2026,nhc_harvey_2018}. The consequences of lost access
are also unevenly distributed. Hospital closures and increased travel distance
are associated with adverse outcomes, particularly in rural and underserved
communities with few nearby alternatives
\citep{buchmueller_how_2006,khushalani_impact_2023}. Tropical cyclones can
simultaneously increase healthcare demand and restrict access, with socially
vulnerable communities facing greater barriers to obtaining care
\citep{parks_tropical_2021,rickless_social_2023}. Plans based only on aggregate
monetary losses may therefore understate the distributional consequences of
healthcare service disruptions.

Research on disaster operations has examined disaster preparedness~\citep{uichanco_model_2022}, patient evacuation~\citep{KIM2024104518}, and post-disaster recovery~\citep{castro_markov_2025}. However, these models primarily focus on short-term operational decisions once a disruption is imminent or has already occurred, while long-term mitigation planning, particularly in the context of flood events, remains comparatively underexplored~\citep{gupta_disaster_2016}. Related work in other lifeline infrastructures, including power grids and
transportation systems, has increasingly addressed long-term resilience
planning under flood uncertainty
\citep{TOPLUTUTAY2024102036,shukla_co-optimization_2025,
austgen_comparisons_2025,phouratsamay_strategic_2024}. More broadly, strategic mitigation under
uncertain disruptions has been studied through inventory positioning,
facility protection, recovery, and recourse decisions
\citep{gao_disruption_2019,losada_optimizing_2012,
starita_improving_2022,ansari_two-stage_2025}.

Healthcare-fortification studies have considered probabilistic facility
failures, accessibility, investment portfolios, and adversarial interdiction
\citep{liberatore_hedging_2012,aliakbarian_bi-level_2015, toplu-tutay_scenario-based_2022, khanduzi_bilevel_2024,baret_enhancing_2025,baret_optimization_2026}. Nevertheless, prior work has not jointly modeled climate-informed and
spatially dependent flood exposure, permanent facility hardening,
endogenous patient evacuation, and the trade-off between economic loss and service disruption impact. Although climate uncertainty has been incorporated into long-term flood-protection planning
\citep{pol_optimal_2014,postek_adjustable_2019,
phouratsamay_strategic_2024}, its integration with healthcare facility
protection and distributional disruption impacts remains limited.

We address this gap using a bi-objective two-stage stochastic mixed-integer
program with climate-informed tropical-cyclone flood scenarios. The first
stage selects discrete hardening levels subject to a capital budget. After a
scenario is realized, the recourse problem identifies disrupted facilities and
assigns their patients to safe receivers with available capacity. The two
objectives minimize expected economic loss and service disruption impact, respectively.

Even the single-objective problem of minimizing expected economic loss is
NP-hard. The full bi-objective formulation compounds this computational
challenge through scenario-dependent facility states and dense capacitated
evacuation networks. To solve the resulting large-scale problem, we develop an
exact Benders decomposition and a scalable Lagrangian-dual method with tailored
primal recovery. We embed both methods in an adaptive
$\varepsilon$-constraint procedure that concentrates computational effort on
informative regions of the Pareto frontier.

This study makes four contributions. First, it integrates proactive facility
hardening and scenario-dependent patient evacuation in a long-term healthcare facility resilience model. Second, it introduces a service-disruption impact measure that combines the duration and scale of lost healthcare capacity with a social-vulnerability weight. Third, it combines a real-world dataset of 3,752 Texas hospitals and nursing homes with realistic climate-informed future flood scenarios. Fourth, it develops exact and approximate decomposition methods for constructing the Pareto frontier of the
large-scale bi-objective stochastic program.

%% file: 02_model.tex
\section{Problem description and mathematical formulation}
\label{sec:model}

The stochastic optimization model is formulated from the perspective of a single central decision-maker (e.g., a public agency or regional authority) who jointly determines mitigation investments and allocates disaster relief funds following flood events. Hardening is installed before the flood scenario is observed and represents a permanent mitigation measure such as a perimeter floodwall. After scenario $s\in\mathcal{S}$ is realized, facilities whose flood depth exceeds their hardening level become non-operational and their occupied beds must be evacuated to safe facilities with available capacity. The second stage therefore determines operational responses, including patient
evacuation, and the resulting economic and disruption impacts. Table~\ref{tbl:Notation_healthcaresys} summarizes the central notation.

\begin{table}[!t]
\centering
\caption{Sets, parameters, and decision variables}
\label{tbl:Notation_healthcaresys}

\footnotesize
\setstretch{1.0}
\setlength{\tabcolsep}{4pt}
\renewcommand{\arraystretch}{1.05}

\begin{tabularx}{\textwidth}{
    @{}
    >{\raggedright\arraybackslash}p{0.22\textwidth}
    >{\raggedright\arraybackslash}X
    @{}
}
\toprule
Symbol & Description \\
\midrule

\multicolumn{2}{@{}l}{\textit{Sets}} \\
\addlinespace[2pt]

$J$
& Set of healthcare facilities (hospitals and nursing homes). \\

$\mathcal{S}$
& Set of flooding scenarios. \\

$J^{\text{send}}$
& Subset of facilities that experience positive flood depth in at least one
scenario:
$\left\{j\in J \mid \exists s\in\mathcal{S}\text{ such that }F_{sj}>0\right\}$. \\

$J^{\text{recv}}_{s}$
& Subset of facilities that can safely receive evacuees in scenario $s$:
$\left\{k\in J \mid F_{sk}=0,\;A_k>0\right\}$. \\

$\mathcal{A}_s$
& Scenario-dependent evacuation arc set:
$\left\{(j,k)\in J^{\text{send}}\times J^{\text{recv}}_s \mid k\neq j\right\}$. \\
\midrule

\addlinespace[4pt]
\multicolumn{2}{@{}l}{\textit{Parameters}} \\
\addlinespace[2pt]

$C_j^{H}$
& Cost of permanent hardening per level of hardening for facility $j$. \\

$C_{jk}^{E}$
& Evacuation cost per patient from facility $j$ to facility $k$. \\

$R_j$
& Restoration and business-interruption cost per unit of excess flood depth
at facility $j$. \\

$V_j$
& Service disruption impact per unit of excess flood depth at facility $j$. \\

$O_j,\;D_j$
& Bed capacity and potential evacuation demand (occupied beds) at facility
$j$, respectively. \\

$A_j=O_j-D_j$
& Available bed capacity at facility $j$. \\

$P_s$
& Probability of scenario $s$. \\

$B$
& Total permanent hardening budget. \\

$F_{sj}$
& Flood level at facility $j$ in scenario $s$, rounded up to the nearest
integer. \\

$H_j^{\max}$
& Maximum allowable hardening level at facility $j$, defined as the maximum
of $F_{sj}$ across all scenarios $s$. \\
\midrule
\addlinespace[4pt]
\multicolumn{2}{@{}l}{\textit{Decision variables}} \\
\addlinespace[2pt]

$y_j\in\{0,\ldots,H_j^{\max}\}$
& Discrete hardening level at facility $j$,
$j\in J^{\text{send}}$. \\

$q_{sjk}\in\mathbb{Z}_{+}$
& Number of patients evacuated from facility $j$ to receiving facility $k$
in scenario $s$, defined for $(j,k)\in\mathcal{A}_s$ and
$s\in\mathcal{S}$. \\

$\gamma_{sj}\in\{0,1\}$
& Equal to 1 if facility $j$ is flooded and non-operational in scenario $s$,
and 0 otherwise, for $j\in J^{\text{send}}$ and $s\in\mathcal{S}$. \\

$z_{sj}\geq 0$
& Excess flood depth above the hardening level at facility $j$ in scenario
$s$, for $j\in J^{\text{send}}$ and $s\in\mathcal{S}$. \\

\bottomrule
\end{tabularx}
\end{table}

\subsection{Objectives}
The expected economic-loss objective includes evacuation expenses and post-flood restoration and business interruption costs:
\begin{align}
 f_1(y,z,q)
 = \sum_{s\in\mathcal{S}} P_s
 \left(
 \sum_{(j,k)\in\mathcal{A}_s} C_{jk}^E q_{sjk}
 + \sum_{j\in J^{\mathrm{send}}} R_j z_{sj}
 \right).
 \label{eq:f1}
\end{align}

Monetary losses alone, however, do not capture how the consequences of
lost healthcare capacity are distributed across communities. Hospital closures
and the resulting increases in travel distance can reduce access to care and
adversely affect health outcomes
\citep{buchmueller_how_2006,khushalani_impact_2023}, while communities with
greater social vulnerability face additional barriers to healthcare access
\citep{alrifai_state-level_2022,rickless_social_2023}.
More generally, empirical
evidence indicates that infrastructure-service disruptions can produce
heterogeneous well-being impacts across sociodemographic groups and motivates
human-centric resilience measures that account for both disruption exposure
and population vulnerability \citep{dargin_human-centric_2020}. Because patient-origin, facility-service-area, and patient-level outcome data
are unavailable, we represent the distributional dimension of service loss
using an expected, location-based, vulnerability-weighted service-disruption impact index:
\begin{align}
 f_2(y,z)
 = \sum_{s\in\mathcal{S}}P_s
 \sum_{j\in J^{\mathrm{send}}}V_j z_{sj}.
 \label{eq:f2}
\end{align}
We define $V_j=I_j\alpha^{t(j)}\varphi_j$, where $I_j$ is a
place-based social-vulnerability weight corresponding to the census tract
containing facility $j$, $\alpha^{t(j)}$ is restoration time per unit excess depth for facility type $t(j)$, and $\varphi_j$ represents facility scale.  Thus, $V_jz_{sj}$ captures the duration and scale of facility $j$'s service disruption in scenario $s$, while assigning greater weight to facilities located in more socially vulnerable areas.

For each fixed budget, we apply the $\varepsilon$-constraint method
\citep{ehrgott2005multicriteria,cohon1978multiobjective,
mavrotas2009effective}:
\begin{align}
 \nonumber \min \quad & f_1(y,z,q) \\
 \text{s.t.}\quad & f_2(y,z)\leq \varepsilon.
 \label{eq:epsilon}
\end{align}

For each fixed budget level, the admissible interval for $\varepsilon$ is
defined using lexicographically refined payoff-table endpoints. We first
minimize $f_1$ and then, among all solutions attaining $f_1^{\min}$,
minimize $f_2$. This yields the Pareto-efficient endpoint
$(f_1^{\min},f_2^{\max})$. Analogously, we first minimize $f_2$ and then,
among all solutions attaining $f_2^{\min}$, minimize $f_1$, yielding
$(f_1^{\max},f_2^{\min})$. The resulting interval
$[f_2^{\min},f_2^{\max}]$ excludes weakly efficient endpoints and defines
the range over which the trade-off between economic loss and service
disruption impact is evaluated. A schematic illustration of the lexicographic refinement is provided in
Figure~\ref{fig:conceptual_payoff_endpoints}.
\begin{figure}[H]
    \centering
    \includegraphics[width=0.70\textwidth]{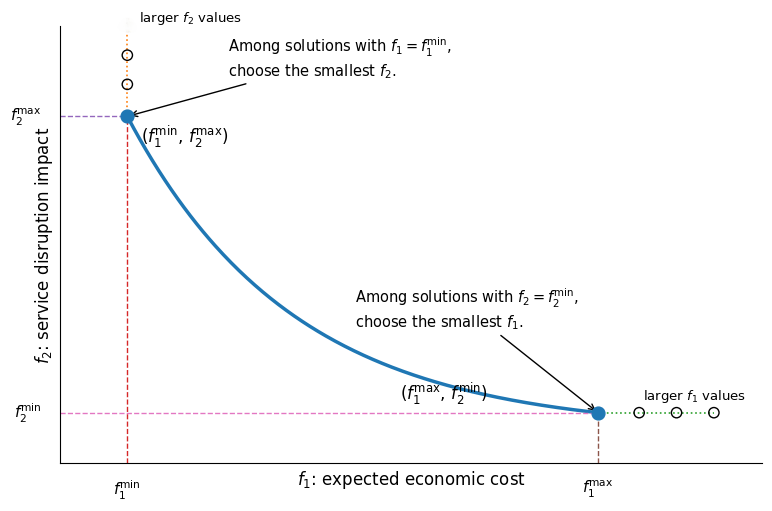}
    \caption{Conceptual illustration of the payoff-table endpoints and lexicographic tie-breaking. The open circles represent alternative end point solutions. The figure is schematic and not drawn to scale.}
    \label{fig:conceptual_payoff_endpoints}
\end{figure}

\subsection{Feasibility constraints}

Hardening expenditures cannot exceed the available budget:
\begin{align}
 \sum_{j\in J^{\mathrm{send}}} C_j^H y_j \leq B,
 \label{eq:budget}
\end{align}
where $B$ is an exogenous policy parameter reflecting funding limits imposed by public agencies or grant programs.

Excess flood depth and operational status satisfy
\begin{align*}
 z_{sj} &\geq F_{sj}-y_j,
 &&\forall s\in\mathcal{S},\ j\in J^{\mathrm{send}}, \\
 z_{sj} &\leq F_{sj}\gamma_{sj},
 &&\forall s\in\mathcal{S},\ j\in J^{\mathrm{send}}, \\
 z_{sj} &\geq 0,
 &&\forall s\in\mathcal{S},\ j\in J^{\mathrm{send}}.
\end{align*}
At optimality, these constraints imply $z_{sj}=(F_{sj}-y_j)_+$ and $\gamma_{sj}=\mathbf{1}[F_{sj}>y_j]$.

If a facility is non-operational, all occupied beds must be evacuated:
\begin{align*}
 \sum_{k:(j,k)\in\mathcal{A}_s} q_{sjk}
 =D_j\gamma_{sj},
 \quad \forall s\in\mathcal{S},\ j\in J^{\mathrm{send}}.
\end{align*}
Receiving capacities impose
\begin{align}
 \sum_{j:(j,k)\in\mathcal{A}_s}q_{sjk}
 \leq A_k,
 \quad \forall s\in\mathcal{S},\ k\in J_s^{\mathrm{recv}}.
 \label{eq:recv_capacity}
\end{align}
The decision domains are $y_j\in\{0,\ldots,H_j^{\max}\}$, $\gamma_{sj}\in\{0,1\}$, and $q_{sjk}\geq0$. To ensure complete recourse, we include a dummy destination with large capacity and prohibitive transportation cost, guaranteeing feasibility under all scenarios. This dummy node reflects the practical reality that patients can always be evacuated farther away at a high logistical and economic cost.

\subsection{Computational Complexity}
\label{sec:complexity}

The computational challenge of the introduced problem does not arise solely from
its multi-objective structure. Even the single-objective problem that minimizes
expected economic loss, \(f_1\), is NP-hard.

\begin{proposition}
The single-objective \(f_1\) version of the healthcare facility resilience planning problem is NP-hard.
\end{proposition}

\begin{proof}
Consider an arbitrary \(0\)--\(1\) knapsack instance with item weights \(w_j\),
profits \(\rho_j\), and capacity \(W\). Construct a restricted resilience
planning instance with a single scenario of probability one. For every sending
facility \(j\), let
\[
F_{sj}=H_j^{\max}=D_j=1,\qquad
C_j^H=w_j,
\]
and introduce one safe receiving facility \(k\) with sufficient capacity,
\(A_k\ge |J^{\mathrm{send}}|\). Set \(B=W\) and choose nonnegative restoration
and evacuation costs satisfying
\[
R_j+C_{jk}^E=\rho_j.
\]

Because \(F_{sj}=H_j^{\max}=1\), the hardening decision is binary:
\(y_j\in\{0,1\}\). If \(y_j=0\), the flood-status and evacuation constraints
force
\[
z_{sj}=\gamma_{sj}=q_{sjk}=1.
\]
If \(y_j=1\), setting
\[
z_{sj}=\gamma_{sj}=q_{sjk}=0
\]
is feasible and optimal because all associated costs are nonnegative.
Therefore, at optimality,
\[
z_{sj}=\gamma_{sj}=q_{sjk}=1-y_j.
\]
The restricted problem consequently becomes
\[
\min \sum_{j\in J^{\mathrm{send}}}
(R_j+C_{jk}^E)(1-y_j)
\qquad
\text{s.t.}
\qquad
\sum_{j\in J^{\mathrm{send}}} C_j^H y_j\le B,
\qquad
y_j\in\{0,1\}.
\]
Since \(\sum_j(R_j+C_{jk}^E)\) is constant, substituting the constructed
parameter values yields the equivalent problem
\[
\max \sum_{j\in J^{\mathrm{send}}}\rho_j y_j
\qquad
\text{s.t.}
\qquad
\sum_{j\in J^{\mathrm{send}}}w_j y_j\le W,
\qquad
y_j\in\{0,1\},
\]
which is the original \(0\)--\(1\) knapsack problem. This shows that the knapsack problem is a special case of the restricted instance of the resilience-planning problem. Since \(0\)--\(1\) knapsack is NP-hard, the restricted problem, and therefore the general facility resilience-planning problem, is also NP-hard.
\end{proof}

The full formulation additionally includes scenario-dependent facility states,
capacitated evacuation flows, and, under the \(\varepsilon\)-constraint
formulation, coupling across scenarios and facilities through the disruption-impact constraint. Its extensive form therefore grows rapidly with both the number of facilities and the number of flood scenarios. The complexity result rules out
a polynomial-time exact algorithm in general, but the formulation retains substantial exploitable structure: conditional on the first-stage hardening and
facility-status decisions, evacuation decisions separate by scenario and form
polynomial-time transportation problems. This structure motivates the decomposition and relaxation methods developed next.

%% file: 03_solution_approaches.tex
\section{Solution approaches}
\label{sec:solution}

We refer to solving the full formulation directly with an optimization solver
as the extensive approach. We use an exact multi-cut Benders decomposition as the primary exact solution method and a Lagrangian-dual
method as a scalable approximation. Each approach solves a sequence of $\varepsilon$-constrained problems to construct the Pareto frontier.

\subsection{Exact multi-cut Benders decomposition}
\label{sec:benders}

A direct decomposition that assigns all second-stage variables and linking
constraints to the subproblems would produce mixed-integer subproblems whose
value functions are generally nonconvex in the first-stage decisions.
Classical linear Benders optimality cuts would therefore not be available.
To preserve the L-shaped structure, the master problem contains the hardening variables $y$, facility-status
variables $\gamma$, excess-depth variables $z$, and one recourse-cost
approximation variable $\beta_s$ for each scenario. It also retains the
original budget, disruption-impact, and flood-damage constraints together with
the accumulated scenario-specific Benders cuts. Let \(\mathcal C_s\) denote
the current cut set for scenario \(s\). The master problem is
\begin{align}
\min_{y,\gamma,z,\beta}\quad
& \sum_{s\in\mathcal S}P_s
\left(\sum_{j\in J^{\mathrm{send}}}R_jz_{sj}+\beta_s\right)
\label{MP:obj}
\\
\text{s.t.}\quad
& \sum_{s\in\mathcal S}P_s
\sum_{j\in J^{\mathrm{send}}}V_jz_{sj}\leq\varepsilon,
\label{MP:epsilon}
\\
& \sum_{j\in J^{\mathrm{send}}}C_j^Hy_j\leq B,
\label{MP:budget}
\\
& z_{sj}\geq F_{sj}-y_j,
&& \forall s\in\mathcal S,\ j\in J^{\mathrm{send}},
\label{MP:z_lb}
\\
& z_{sj}\leq F_{sj}\gamma_{sj},
&& \forall s\in\mathcal S,\ j\in J^{\mathrm{send}},
\label{MP:z_ub}
\\
& \beta_s\geq
\sum_{j\in J^{\mathrm{send}}}\zeta_{sj}^{r}D_j\gamma_{sj}
+\sum_{k\in J_s^{\mathrm{recv}}}\chi_{sk}^{r}A_k,
&& \forall r\in\mathcal C_s,\ s\in\mathcal S,
\label{MP:cuts}
\\
& y_j\in\{0,\ldots,H_j^{\max}\},
&& \forall j\in J^{\mathrm{send}},
\notag\\
& \gamma_{sj}\in\{0,1\},\quad z_{sj}\geq0,
&& \forall s\in\mathcal S,\ j\in J^{\mathrm{send}},
\notag\\
& \beta_s\geq0,
&& \forall s\in\mathcal S.
\notag
\end{align}
Here \((\zeta^r,\chi^r)\) denotes an extreme-point dual solution generated
in an earlier iteration. Initially, each \(\mathcal C_s\) is empty. At iteration $t$, the master problem yields a candidate solution
$(y^t,z^t,\gamma^t,\beta^t)$. For fixed facility-status decisions
$\gamma_s^t$, the scenario-$s$ evacuation subproblem is
\begin{align}
\theta_s(\gamma_s^t)
=
\min_q\quad
& \sum_{(j,k)\in\mathcal{A}_s}
C_{jk}^E q_{sjk}
\label{SP:obj}
\\[3pt]
\text{s.t.}\quad
& \sum_{k:(j,k)\in\mathcal{A}_s}
q_{sjk}
=
D_j\gamma_{sj}^t,
&& \forall j\in J^{\mathrm{send}},
\label{SP:supply}
\\
& \sum_{j:(j,k)\in\mathcal{A}_s}
q_{sjk}
\leq A_k,
&& \forall k\in J_s^{\mathrm{recv}},
\label{SP:capacity}
\\
& q_{sjk}\geq0,
&& \forall (j,k)\in\mathcal{A}_s.
\label{SP:domain}
\end{align}

This subproblem is a capacitated transportation problem. Its constraint
matrix is totally unimodular; therefore, integer demands and receiving
capacities yield integral extreme-point solutions without imposing
integrality restrictions on $q$.

Let $\zeta_{sj}$ be the unrestricted dual variable associated with
constraint~\eqref{SP:supply}, and let $\chi_{sk}\leq0$ be the dual variable
associated with constraint~\eqref{SP:capacity}. The corresponding dual is
\begin{align}
\max_{\zeta,\chi}\quad
& \sum_{j\in J^{\mathrm{send}}}\zeta_{sj}D_j\gamma_{sj}^t
+\sum_{k\in J_s^{\mathrm{recv}}}\chi_{sk}A_k
\label{DSP:obj}
\\
\text{s.t.}\quad
& \zeta_{sj}+\chi_{sk}\leq C_{jk}^E,
&& \forall (j,k)\in\mathcal A_s,
\label{DSP:arc}
\\
& \chi_{sk}\leq0,
&& \forall k\in J_s^{\mathrm{recv}}.
\label{DSP:sign}
\end{align}
Let
$(\zeta_{sj}^t,\chi_{sk}^t)$ be an optimal dual solution. By strong duality,
the exact recourse cost is
\begin{align}
\theta_s(\gamma_s^t)
=
\sum_{j\in J^{\mathrm{send}}}
\zeta_{sj}^tD_j\gamma_{sj}^t
+
\sum_{k\in J_s^{\mathrm{recv}}}
\chi_{sk}^tA_k.
\label{CUT:recourse_value}
\end{align}

A cut is generated whenever
\begin{align}
\beta_s^t
<
\theta_s(\gamma_s^t)-\kappa,
\label{CUT:violation_test}
\end{align}
where $\kappa>0$ is a numerical tolerance. The corresponding scenario-specific
Benders optimality cut is
\begin{align}
\beta_s
\geq
\sum_{j\in J^{\mathrm{send}}}
\zeta_{sj}^tD_j\gamma_{sj}
+
\sum_{k\in J_s^{\mathrm{recv}}}
\chi_{sk}^tA_k.
\label{CUT:multicut}
\end{align}

The cut is a valid supporting hyperplane of the scenario recourse function
and is tight at $\gamma_s^t$. Separate violated cuts are added for all
scenarios at each iteration. As the cut pools expand, the master problem
progressively internalizes the evacuation costs associated with alternative
facility-operationality patterns.

The dummy receiving destination introduced in the model guarantees complete
recourse for every feasible master solution; hence, feasibility cuts are not
required. The algorithm terminates when no scenario violates
condition~\eqref{CUT:violation_test}. Because the master contains finitely
many discrete operationality patterns and the subproblems generate exact
linear optimality cuts, the procedure converges finitely to an optimal
solution of each $\varepsilon$-constrained problem.

Although both the extensive formulation and the Benders master problem remain
mixed-integer programs, the decomposition removes the dense evacuation-flow
networks from the master and transfers them to polynomial-time transportation
subproblems. To reduce repeated model-construction overhead, we retain the Benders master and scenario subproblem models throughout the Pareto-frontier sweep and update
only the right-hand side of the $\varepsilon$-constraint between solves. We do not provide explicit MIP starts. The scenario LPs are reoptimized after right-hand-side updates using dual simplex.

\subsection{Scalable Lagrangian-dual method}
\label{sec:lagrangian_dual}

To complement the exact Benders method, we develop a Lagrangian-dual (LD)
approach that relaxes the
service-disruption-impact constraint~\eqref{eq:epsilon}, the hardening-budget constraint~\eqref{eq:budget}, and the receiving-capacity constraints~\eqref{eq:recv_capacity}. The resulting problem decomposes by facility and can be solved by enumerating a small set
of candidate hardening levels. Projected subgradient ascent produces lower
bounds, while a tailored primal-recovery
procedure constructs feasible solutions and corresponding upper bounds. The
LD method therefore provides a scalable approximation for large instances.

\subsubsection{Restricted evacuation network}
\label{sec:ld_restricted_network}

To reduce repeated evacuation calculations, LD retains only the $K$
lowest-cost feasible receivers for each sender--scenario pair. For
$s\in\mathcal S$ and $j\in J^{\mathrm{send}}$, let
$K_{sj}\subseteq J_s^{\mathrm{recv}}$ denote this receiver set, and define
\begin{align*}
\widetilde J_s^{\mathrm{recv}}
&=
\bigcup_{j\in J^{\mathrm{send}}}K_{sj},
\\
\widetilde{\mathcal A}_s
&=
\left\{
(j,k):
j\in J^{\mathrm{send}},\;
k\in K_{sj},\;
k\neq j
\right\}.
\end{align*}

For fixed \(s\) and \(j\), let \(r_{sj}(k)\) be the rank of receiver \(k\)
in nondecreasing evacuation-cost order. Consider a fixed
\(\varepsilon\)-constrained problem, let
\((y^\star,\gamma^\star,z^\star,q^\star)\) be an optimal full-network
solution, and define
\begin{equation}
K^\star=\max\{r_{sj}(k):q_{sjk}^\star>0\}.
\label{eq:kstar}
\end{equation}

\begin{proposition}
If \(K\geq K^\star\), the restricted- and full-network formulations have the
same optimal objective value.
\end{proposition}

\begin{proof}
The restricted feasible region is contained in the full-network feasible
region, so its optimum cannot be smaller. If \(K\geq K^\star\), every arc used
by \(q^\star\) is retained; hence the full-network optimum remains feasible in
the restricted formulation, whose optimum cannot be larger. The two optimal
values are therefore equal.
\end{proof}

For a budget \(B\) and evaluated threshold set \(\mathcal E_B\), taking
\(K\geq\max_{\varepsilon\in\mathcal E_B}K^\star(\varepsilon)\) preserves at
least one full-network optimum at every evaluated frontier point. This
sufficient threshold is not observable before solving the full-network
problems and may equal the full receiver set in the worst case because
facilities compete for limited capacity. We therefore select \(K\) through
validation against full-network Benders solutions. Unless exactness has been
established, the LD lower bound applies to the restricted-network problem;
any feasible solution recovered on the restricted network is also feasible
for the original full-network problem.

\subsubsection{Lagrangian relaxation and facility-wise decomposition}
\label{sec:ld_relaxation}

Let $\mu\geq0$ be the multiplier associated with the hardening-budget
constraint, $\pi_{sk}\geq0$ the multiplier associated with receiving capacity
at facility $k$ in scenario $s$, and $\nu\geq0$ the multiplier associated
with the service-disruption-impact constraint. Because
$K_{sj}\subseteq J_s^{\mathrm{recv}}$, every retained receiver is operational
in scenario $s$ and has available capacity $A_k$. For fixed multipliers, the
Lagrangian relaxation is
\begin{align*}
L(\mu,\pi,\nu)
:=
\min_{y,z,\gamma,q}\quad
&
\sum_{s\in\mathcal S}P_s
\left(
\sum_{j\in J^{\mathrm{send}}}R_jz_{sj}
+
\sum_{(j,k)\in\widetilde{\mathcal A}_s}
C_{jk}^Eq_{sjk}
\right)
\notag\\
&+
\mu
\left(
\sum_{j\in J^{\mathrm{send}}}C_j^Hy_j-B
\right)
\notag\\
&+
\sum_{s\in\mathcal S}
\sum_{k\in\widetilde J_s^{\mathrm{recv}}}
\pi_{sk}
\left(
\sum_{j:(j,k)\in\widetilde{\mathcal A}_s}q_{sjk}
-A_k
\right)
\notag\\
&+
\nu
\left(
\sum_{s\in\mathcal S}P_s
\sum_{j\in J^{\mathrm{send}}}V_jz_{sj}
-\varepsilon
\right)
\\
\text{s.t.}\quad
&
z_{sj}\geq F_{sj}-y_j,
&&
\forall s\in\mathcal S,\;
j\in J^{\mathrm{send}},
\notag\\
&
z_{sj}\leq F_{sj}\gamma_{sj},
&&
\forall s\in\mathcal S,\;
j\in J^{\mathrm{send}},
\notag\\
&
\sum_{k\in K_{sj}}q_{sjk}
=
D_j\gamma_{sj},
&&
\forall s\in\mathcal S,\;
j\in J^{\mathrm{send}},
\notag\\
&
y_j\in\{0,\ldots,H_j^{\max}\},
&&
\forall j\in J^{\mathrm{send}},
\notag\\
&
\gamma_{sj}\in\{0,1\},\quad z_{sj}\geq0,
&&
\forall s\in\mathcal S,\;
j\in J^{\mathrm{send}},
\notag\\
&
q_{sjk}\geq0,
&&
\forall s\in\mathcal S,\;
(j,k)\in\widetilde{\mathcal A}_s.
\notag
\end{align*}
For every $(\mu,\pi,\nu)\geq0$, $L(\mu,\pi,\nu)$ is a valid lower bound on
the optimal value of the restricted-network $\varepsilon$-constraint problem.

For fixed multipliers, the relaxed problem decomposes by sending facility.
Define
\begin{equation*}
W_j(\nu)=R_j+\nu V_j,
\end{equation*}
and, for a fixed hardening level $y_j$,
\begin{equation*}
\gamma_{sj}(y_j)=\mathbf 1[F_{sj}>y_j],
\qquad
z_{sj}(y_j)=(F_{sj}-y_j)_+.
\end{equation*}
Because hosting-capacity constraints are relaxed, evacuees from origin $j$ in scenario $s$ are routed independently to the cheapest destination under the modified costs
\begin{equation}
m_{sj}(\pi)
=
\min_{k\in K_{sj}}
\left\{
P_sC_{jk}^E+\pi_{sk}
\right\},
\label{LR:modified_cost}
\end{equation}
where $P_s C^E_{jk}+\pi_{sk}$ is the modified per-unit transportation cost. The facility-level contribution is therefore
\begin{align*}
\phi_j(y_j;\mu,\pi,\nu)
=
\mu C_j^Hy_j
+
\sum_{s\in\mathcal S:\,F_{sj}>y_j}
\left[
D_jm_{sj}(\pi)
+
P_sW_j(\nu)(F_{sj}-y_j)
\right],
\end{align*}
and the Lagrangian value is
\begin{align*}
L(\mu,\pi,\nu)
=
\sum_{j\in J^{\mathrm{send}}}
\min_{y_j\in\mathcal Y_j}
\phi_j(y_j;\mu,\pi,\nu)
-
\sum_{s\in\mathcal S}
\sum_{k\in\widetilde J_s^{\mathrm{recv}}}
\pi_{sk}A_k
-\mu B-\nu\varepsilon.
\end{align*}

The function $\phi_j$ is piecewise linear in $y_j$, with breakpoints at the
scenario flood depths. Hence, it suffices to evaluate the distinct values in
\begin{equation}
\mathcal Y_j
=
\{0\}
\cup
\{F_{sj}:s\in\mathcal S\}.
\label{LR:candidate_levels}
\end{equation}
This set contains at most $|\mathcal S|+1$ candidate values, although many
values coincide in practice. To evaluate them efficiently, define
\begin{equation*}
\vartheta_{sj}=P_sW_j(\nu),
\qquad
\upsilon_{sj}=P_sW_j(\nu)F_{sj}+D_jm_{sj}(\pi).
\end{equation*}
For each facility \(j\), sort the scenarios so that
\(F_{(1)j}\leq\cdots\leq F_{(|\mathcal S|)j}\), and construct the suffix sums
\begin{equation*}
\Theta_{ij}=\sum_{r=i}^{|\mathcal S|}\vartheta_{(r)j},
\qquad
\Upsilon_{ij}=\sum_{r=i}^{|\mathcal S|}\upsilon_{(r)j}.
\end{equation*}
If \(F_{(i-1)j}\leq y_j<F_{(i)j}\), the active scenarios are
\(i,\ldots,|\mathcal S|\), and
\begin{equation}
\phi_j(y_j;\mu,\pi,\nu)
=\mu C_j^Hy_j+\Upsilon_{ij}-y_j\Theta_{ij}.
\label{eq:suffix_phi}
\end{equation}
Thus, after sorting once, all candidate levels can be evaluated in linear
time for each facility. Consequently, each iteration of the Lagrangian
subproblem can be solved in
\[
O\!\left(
|J^{\mathrm{send}}|\,|\mathcal S|\,K
\right)
\]
time. Pseudocode for the \textsc{EnumerateHardening} subroutine is provided
in the Supplementary Material.

\subsubsection{Dual optimization and tailored primal recovery}
\label{sec:ld_dual_recovery}

The Lagrangian dual problem is
\[
\max_{\mu\geq0,\;\pi\geq0,\;\nu\geq0}
L(\mu,\pi,\nu).
\]

Given the Lagrangian minimizer
$(y^t,\gamma^t,z^t,q^t)$ at iteration $t$, the subgradient components are
\begin{align*}
g_\mu^t
&=
\sum_{j\in J^{\mathrm{send}}}C_j^Hy_j^t-B,
\\
g_{sk}^t
&=
\sum_{j:(j,k)\in\widetilde{\mathcal A}_s}
q_{sjk}^t-A_k,
&&
\forall s\in\mathcal S,\;
k\in\widetilde J_s^{\mathrm{recv}},
\\
g_\nu^t
&=
\sum_{s\in\mathcal S}P_s
\sum_{j\in J^{\mathrm{send}}}V_jz_{sj}^t
-\varepsilon.
\end{align*}

We use projected subgradient ascent with the Polyak step size
\[
\eta^t
=
\theta
\frac{UB-L(\mu^t,\pi^t,\nu^t)}
{\lVert g^t\rVert_2^2},
\qquad
\theta\in(0,2),
\]
where $UB$ is the best feasible upper bound available at iteration $t$.
The multipliers are updated as
\begin{align*}
\mu^{t+1}
&=
[\mu^t+\eta^tg_\mu^t]_+,
\\
\pi_{sk}^{t+1}
&=
[\pi_{sk}^t+\eta^tg_{sk}^t]_+,
\\
\nu^{t+1}
&=
[\nu^t+\eta^tg_\nu^t]_+.
\end{align*}

The Lagrangian minimizer need not satisfy the relaxed constraints. At each
iteration, a tailored primal-recovery procedure therefore:
(i) removes hardening units until the budget constraint is satisfied
(\textsc{BudgetRepair});
(ii) uses the remaining budget to add hardening units until
$f_2\leq\varepsilon$ (\textsc{EpsilonRepair}); and
(iii) if the first two steps succeed, constructs evacuation flows satisfying
the receiving-capacity constraints on the restricted network
(\textsc{GreedyTransport}).

For facility $j$ and hardening increment
$\ell\in\{1,\ldots,H_j^{\max}\}$, define
\begin{align*}
\Delta_j^{\mathrm{imp}}(\ell)
&=
\sum_{s\in\mathcal S:\,F_{sj}\geq\ell}
P_sV_j,
\\
\Delta_j^{\mathrm{rest}}(\ell)
&=
\sum_{s\in\mathcal S:\,F_{sj}\geq\ell}
P_sR_j.
\end{align*}
These quantities measure the marginal reductions in service-disruption impact
and expected restoration loss, respectively, obtained by increasing the
hardening level from $\ell-1$ to $\ell$. Hardening additions and removals are
ranked using
\[
\rho_j(\ell)
=
\frac{
(1-\omega)
\dfrac{\Delta_j^{\mathrm{imp}}(\ell)}
{\Delta_{\max}^{\mathrm{imp}}}
+
\omega
\dfrac{\Delta_j^{\mathrm{rest}}(\ell)}
{\Delta_{\max}^{\mathrm{rest}}}
}
{C_j^H},
\qquad
\omega\in[0,1],
\]
where
\[
\Delta_{\max}^{\mathrm{imp}}
=
\max_{j,\ell}\Delta_j^{\mathrm{imp}}(\ell),
\qquad
\Delta_{\max}^{\mathrm{rest}}
=
\max_{j,\ell}\Delta_j^{\mathrm{rest}}(\ell).
\]
\textsc{BudgetRepair} removes hardening units in increasing order of
$\rho_j(\ell)$, whereas \textsc{EpsilonRepair} adds units in decreasing
order.

Given the repaired hardening plan, \textsc{GreedyTransport} processes
evacuating facilities in decreasing order of demand and assigns their demand
to receivers in increasing order of evacuation cost while updating residual
capacities. Recovery succeeds only if the complete demand of every evacuating
facility is assigned. A successful recovery produces a feasible solution and
the upper bound
\[
UB^t
=
\sum_{s\in\mathcal S}P_s
\left(
\sum_{j\in J^{\mathrm{send}}}R_j\bar z_{sj}
+
\sum_{(j,k)\in\widetilde{\mathcal A}_s}
C_{jk}^E\bar q_{sjk}
\right).
\]
Detailed pseudocode for the \textsc{PrimalRecovery} procedure, integrating
the \textsc{BudgetRepair}, \textsc{EpsilonRepair}, and
\textsc{GreedyTransport} subroutines, is provided in the Supplementary
Material.

\subsubsection{Frontier sweep}
\label{sec:ld_frontier}

The $\varepsilon$ values are processed from the tightest to the loosest impact
threshold. Before the first LD iteration at the tightest threshold, a constructive heuristic \textsc{InitialFeasiblePoint} generates an
initial feasible plan by adding hardening units according to impact reduction
per dollar, using restoration reduction as a tie-breaker.  At subsequent
frontier points, the final multipliers and best feasible plan from the
preceding point are used as warm starts. Because a plan feasible for a tighter
impact threshold remains feasible for a looser threshold, the previous
incumbent supplies an immediate upper bound.

The complete LD algorithm, including the $\varepsilon$-sweep procedure and
detailed pseudocode for \textsc{InitialFeasiblePoint}, is provided in the
Supplementary Material.

For each threshold, we report
\begin{equation}
\mathrm{gap}
=
\frac{UB-LB}{\max\{1,UB\}}.
\label{LR:reported_gap}
\end{equation}
This is a duality gap for the restricted-network formulation. The method
terminates when the prescribed gap tolerance is met or the iteration limit is
reached.

\subsection{Adaptive solution of the multi-objective problem}
\label{sec:adaptive_multiobjective}
Rather than using uniformly spaced thresholds, we generate up to a prescribed
number of interior $\varepsilon$ values through adaptive interval refinement.
Let the currently accepted non-dominated points be ordered by
$\varepsilon_1<\cdots<\varepsilon_{\mathcal M}$. For interval
$[\varepsilon_m,\varepsilon_{m+1}]$, define its primary-objective variation as
\[
\Delta_m^{(1)}
=
\left|
f_1(\varepsilon_{m+1})-f_1(\varepsilon_m)
\right|.
\]

For each interior point, use the finite-difference curvature proxy
\[
\kappa_m
=
\frac{
\left|
f_1(\varepsilon_{m+1})
-2f_1(\varepsilon_m)
+f_1(\varepsilon_{m-1})
\right|
}{
[(\varepsilon_{m+1}-\varepsilon_{m-1})/2]^2
},
\qquad
m=2,\ldots,\mathcal M-1,
\]
and assign interval $m$ the curvature score
$\Delta_m^{(2)}=\max\{\kappa_m,\kappa_{m+1}\}$, using the available
neighboring value for a boundary interval. An interval is eligible when
either score lies at or above its empirical $q$-quantile across the current
intervals.

The highest-priority eligible interval is evaluated at its midpoint, provided
that the candidate is separated from all previously attempted thresholds by
at least
\[
\tau_\varepsilon
=
\alpha\,Q_{0.5}
\left(
\{\varepsilon_{r+1}-\varepsilon_r\}_{r=1}^{\mathcal M-1}
\right),
\]
where $Q_{0.5}$ denotes the empirical median. If the midpoint produces a
duplicate or dominated objective point, the interval is marked as a plateau
and excluded from further refinement; otherwise, the new point is inserted
into the ordered set. Refinement stops when the prescribed number of interior
points is reached or no admissible interval remains. 

For each fixed budget level, the adaptive procedure is executed once to
generate a common set of $\varepsilon$ values. The same set is then used for
the extensive-form approach, Benders decomposition, and the LD method,
ensuring comparison on identical $\varepsilon$-constrained instances. The
one-time adaptive generation is excluded from the method-specific runtime
comparisons.

%% file: 04_case_study.tex
\section{Case study and experimental design}
\label{sec:case_study_experimental_design}

\subsection{Study system and facility data}
\label{sec:facility_data}

We study flood-resilience investments for hospitals and nursing homes across
Texas. Facility locations and licensed bed counts are obtained from archived
2022 versions of the Homeland Infrastructure Foundation-Level Data
(HIFLD).\footnote{\url{https://www.arcgis.com/home/item.html?id=f36521f6e07f4a859e838f0ad7536898}
and
\url{https://www.arcgis.com/home/item.html?id=923e9f63e2bb4877bf58db1cd1dd1d82}.
These datasets are no longer publicly available in the form used here.}
After removing closed facilities and records without bed-count information,
the study network contains $3{,}752$ facilities: $761$ hospitals and $2{,}991$
nursing homes.

Let $O_j$ denote the number of beds at facility $j$. Potential evacuation
demand is estimated as
$D_j=\left\lceil O_j\rho_j\right\rceil$, where $\rho_j$ is the assumed bed-occupancy rate, set to $0.75$ for hospitals and $0.63$ for nursing homes
\citep{cdc_hospital_occupancy,tx_hhsc_nursing_occupancy}. The remaining capacity available to receive patients is $A_j=O_j-D_j$. Of the full facility set, $122$ facilities experience positive flood depth in
at least one scenario and form $J^{\mathrm{send}}$; these include $23$
hospitals and $99$ nursing homes. 

\subsection{Flood scenario construction}
\label{sec:scenario_generation}

We construct facility-level scenarios of tropical-cyclone-induced fluvial
flooding following established hydrologic and climate-impact modeling
approaches
\citep{johnson_integrated_2019,kim_hurricane_2021,
maraun_statistical_2018}, with implementation details
adapted from \citet{Tabassum2024Dissertation}. The historical event set contains
27 tropical cyclones from 1993--2018, corresponding to the National Water Model (NWM) Reanalysis
period. We include storms that either made landfall in Texas or produced
substantial rainfall within a 60-mile buffer of the state boundary. Historical
storm precipitation is perturbed using future extreme-precipitation projections
from an ensemble of 15 Coupled Model Intercomparison Project Phase 6 (CMIP6) global climate models
\citep{ONeill2016SSP}.

The main scenario set, \Smain, contains 27 storm realizations based on the
median Shared Socioeconomic Pathway 5–8.5 (SSP5-8.5) precipitation response for 2025--2050, representing a
high-emissions, fossil-fuel--intensive pathway. Across exposed facilities,
the resulting maximum required protection levels $H_j^{\max}$ range from 2 to 20 ft, with an average of approximately 8 ft. The scenarios are assigned equal probabilities because the historical storm set is treated as an empirical sample of tropical cyclones affecting Texas. For computational stress testing,
\Spercentile\ applies the 25th, 50th, and 75th percentile CMIP6 responses to
each storm, yielding 81 scenarios and more severe exposure, with maximum required protection levels ranging from 3 to 44 ft. The percentile weights are
$(0.25,0.50,0.25)$, respectively. The historical storm list is provided in the Supplementary Material.

\subsection{Model parameterization}
\label{sec:case_parameterization}

All monetary parameters are expressed in constant
2024 U.S. dollars. The hardening decision $y_j$ represents feet of permanent flood protection.
Its unit cost is
\begin{equation*}
C_j^H=uX_j,
\qquad
u=\$200
\ \text{per linear foot per foot of protection},
\end{equation*}
where $X_j$ is an effective facility perimeter estimated from bed capacity
using size-dependent building-area data \citep{definitivehc_sqft}. The unit
cost $u$ is selected from FEMA planning-level floodwall cost ranges and
validated against the reported Lourdes Hospital floodwall project
\citep{SCCOG_critical_facilities,fema_advanced_2012}. Additionally, because facility-level information on existing flood protection is unavailable, all facilities are assumed to have zero initial hardening.

The restoration coefficient combines physical damage and business
interruption:
\begin{equation*}
R_j
=
\varphi_j
\left(
r_{\mathrm{phys}}^{t(j)}
+
r_{\mathrm{BI}}^{t(j)}
\right)
=
\begin{cases}
79.5\,\varphi_j, & t(j)=H,\\
42.6\,\varphi_j, & t(j)=N,
\end{cases}
\end{equation*}
where $\varphi_j$ is estimated facility area and $t(j)$ denotes hospital
($H$) or nursing home ($N$). The physical-damage coefficients are derived
from U.S. Army Corps of Engineers depth--damage relationships and Hazus
replacement values
\citep{naccs_depth_damage,hazus_inventory_tm}. Business-interruption
coefficients are derived from Hazus daily income, wage, recapture, and
restoration-time parameters \citep{hazus_flood_tm}.

The service-disruption coefficient is
$V_j=I_j\alpha^{t(j)}\varphi_j$,
where $I_j\in[0,1]$ is the overall CDC/ATSDR Social Vulnerability Index for
the census tract containing facility $j$
\citep{cdc_svi,flanagan_social_2011}. The marginal restoration-time rates,
$\alpha^H=53.4$ and $\alpha^N=51.3$ days/ft, are estimated from Hazus
restoration-time data \citep{hazus_flood_tm}.

Finally, the per-patient evacuation cost is
\begin{equation*}
C_{jk}^E
=
18.45\,d_{jk}+45.9,
\end{equation*}
where $d_{jk}$ is the distance in miles. The distance-dependent component is
based on the City of Houston emergency medical services (EMS) mileage charge,\footnote{\url{https://cohweb.houstontx.gov/FIN_FeeSchedule}}
while the fixed labor component uses U.S. Bureau of Labor Statistics wage
data for a two-person ambulance crew.\footnote{\url{https://www.bls.gov/oes/current/oes292040.htm}}
We assume one hour of crew time per patient transfer, consistent with related
disaster-evacuation studies \citep{KIM2024104518}. Detailed derivations of
facility geometry, hardening, restoration, business interruption,
service-disruption, and evacuation parameters are provided in the
Supplementary Material.

\subsection{Computational design}
\label{sec:comp_setup}
We implement all methods in Python and use Gurobi to solve the mathematical
programs. We solve both the extensive formulation and the Benders master problem with a relative MIP optimality gap of $0.001\%$. The Benders cut-violation tolerance is
$\kappa=10^{-6}$. The LD method uses Polyak parameter $\theta=1.2$ and
primal-recovery weight $\omega=0.9$, and terminates when its reported
relative gap is at most $0.01\%$ or after $1{,}000$ iterations.

The full evacuation network contains approximately $456{,}000$ candidate arcs per
scenario on average. For LD, each sender considers at most $K=110$ candidate receivers, including the dummy receiver. This limits the effective network to at most $122\times110=13{,}420$ arcs per scenario, approximately $97\%$ fewer than in the full network.

The lexicographically refined frontier endpoints, including
$f_2^{\min}$ and $f_2^{\max}$, are computed once to optimality and shared
across all solution methods. $25$ interior $\varepsilon$ values are
generated using the adaptive procedure in
Section~\ref{sec:adaptive_multiobjective}. We set the eligibility
quantile to $q=0.8$ and the minimum-spacing parameter to $\alpha=0.25$. The
resulting fixed set of $\varepsilon$ values is then used for the extensive-form approach, Benders decomposition, and the LD method.

The reported runtime for each method includes all model-construction,
optimization, separation, and recovery operations required to solve the common
set of interior $\varepsilon$-constrained instances. It excludes the shared endpoint computations and the one-time adaptive generation of the
$\varepsilon$ set. All experiments are conducted using single-threaded
execution on an Intel Core i5-8265U CPU (four cores) to ensure consistent
runtime comparisons. 

%% file: 05_results.tex
\section{Results}
\label{sec:results}

\subsection{Economic-loss and service-disruption trade-offs under the main scenario set}

Figure~\ref{fig:pareto_budget_comparison} presents the trade-off between
expected economic cost, $f_1$, and expected service-disruption impact, $f_2$,
under the three hardening budgets.
\begin{figure}[!t]
\centering
\includegraphics[width=0.85\linewidth]{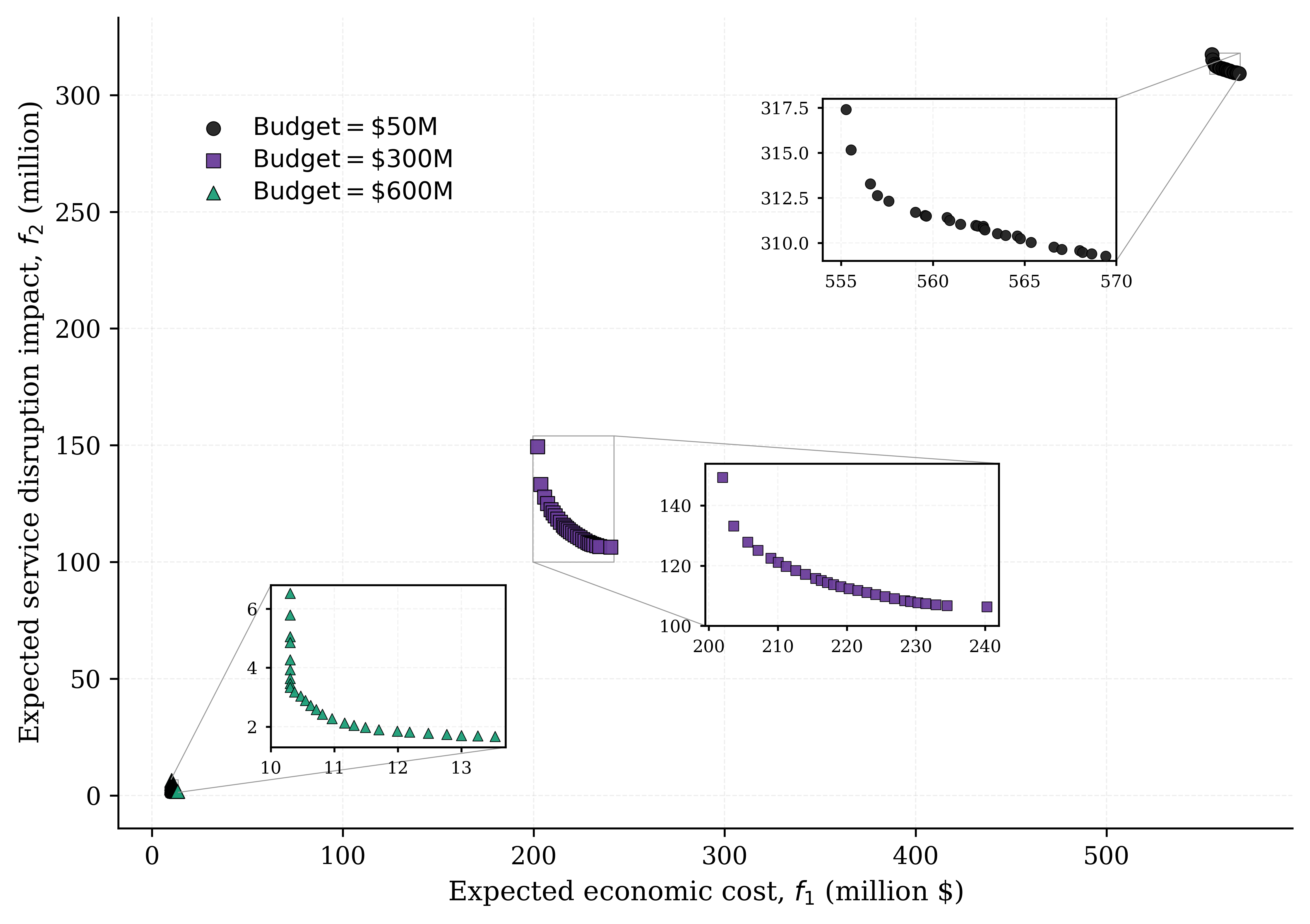}
\caption{Pareto frontiers under low (\$50M), moderate (\$300M), 
and high (\$600M) mitigation budgets. Insets provide zoomed views of the frontier regions to enhance visibility at different scales.}
\label{fig:pareto_budget_comparison}
\end{figure}

Table~\ref{tbl:representative_outcomes} reports the two endpoints and the knee
point for each frontier. The knee is selected as the sampled solution with the
maximum perpendicular distance from the line joining the two endpoints. In
addition to the objective values, the table reports operational measures that
clarify how the Pareto trade-offs translate into evacuation and hardening
decisions.

\begin{table}[!t]
\centering
\small
\setlength{\tabcolsep}{3.5pt}
\renewcommand{\arraystretch}{1.05}
\caption{Objective and operational outcomes for representative Pareto policies. Objective values are reported in millions. Hardening coverage
is the percentage of exposed facilities receiving positive protection, and
average depth is the mean protection level as a percentage of the maximum
required level.}
\label{tbl:representative_outcomes}
\resizebox{\linewidth}{!}{%
\begin{tabular}{llrrrrrr}
\toprule
Budget & Policy & $f_1$ & $f_2$
& \multicolumn{2}{c}{Evacuation demand (patients)}
& Coverage (\%) & Avg. depth (\%) \\
\cmidrule(lr){5-6}
& & & & Expected & Maximum & & \\
\midrule
\multirow{3}{*}{\$50M}
& Impact-min & 569.42 & 309.26 & 344.48 & 6{,}720 & 10.66 & 4.40 \\
& Knee       & 559.05 & 311.70 & 342.37 & 6{,}834 &  9.02 & 4.00 \\
& Cost-min   & 555.27 & 317.40 & 350.85 & 6{,}834 &  9.84 & 3.87 \\
\midrule
\multirow{3}{*}{\$300M}
& Impact-min & 240.25 & 106.35 & 140.78 & 2{,}832 & 42.62 & 37.79 \\
& Knee       & 211.19 & 119.78 & 117.78 & 2{,}781 & 43.44 & 34.61 \\
& Cost-min   & 201.98 & 149.36 &  99.81 & 2{,}275 & 48.36 & 40.22 \\
\midrule
\multirow{3}{*}{\$600M}
& Impact-min & 13.53 & 1.67 & 12.63 & 197 & 95.90 & 93.76 \\
& Knee       & 10.31 & 3.34 &  5.19 & 140 & 93.44 & 90.79 \\
& Cost-min   & 10.31 & 6.53 &  1.93 &  52 & 95.90 & 95.20 \\
\bottomrule
\end{tabular}%
}
\end{table}

The amount of disruption reduction attainable at low economic loss increases
substantially with the hardening budget. At \$600M, moving from the
cost-minimizing solution to the knee reduces $f_2$ by approximately $49\%$,
from $6.53$M to $3.34$M, while increasing $f_1$ by only $0.002\%$.
At \$300M, the corresponding move reduces $f_2$ by approximately $20\%$ at
a $4.6\%$ increase in $f_1$. At \$50M, the improvement is smaller: the knee
reduces $f_2$ by approximately $1.8\%$ while increasing $f_1$ by about
$0.7\%$. Beyond the knee, further reductions in $f_2$ require a sharp increase in $f_1$, indicating strong diminishing returns.

\subsection{How budget availability changes optimal protection}

Figure~\ref{fig:budget_ranges} shows how the attainable ranges of $f_1$ and
$f_2$ change with the hardening budget. Both objectives decline monotonically,
with the largest improvements occurring at low-to-moderate budget levels. The
first \$100M reduces both economic loss and service-disruption impact by more
than half, whereas subsequent increments produce progressively smaller gains.
\begin{figure}[!t]
    \centering
\includegraphics[width=0.85\linewidth]{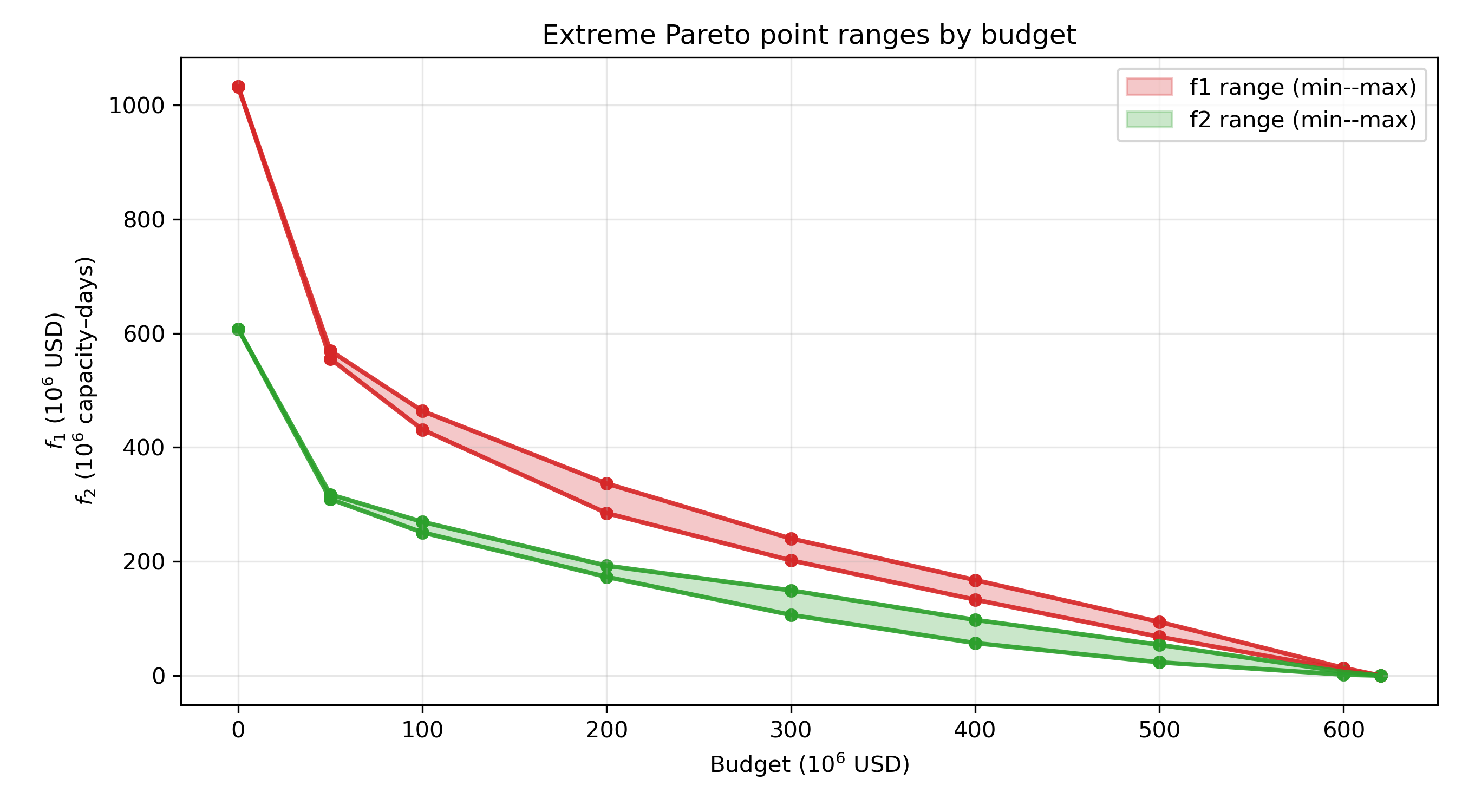}
    \caption{Extreme Pareto point ranges by mitigation budget.
    Shaded bands represent the minimum--maximum achievable values of
    $f_1$ (expected economic cost) and $f_2$ (expected service disruption impact)
    across the Pareto frontier at each budget level.}
    \label{fig:budget_ranges}
\end{figure}

The zero-budget case represents a reactive wait-and-see policy with no
hardening. Across the positive budgets considered, the break-even number of
comparable events required to recover the investment increases from
approximately $0.11$ to $0.60$. Thus, despite diminishing marginal returns,
the modeled expected economic-loss reduction exceeds the investment within
one comparable event at every budget level. Detailed endpoint and break-even
values are reported in the Supplementary Material.

The Pareto ranges are widest at intermediate budgets, where meaningful
trade-offs remain between economic loss and service-disruption impact. They
contract near full hardening: at \$600M,
$f_1\in[10.31,13.53]$M and $f_2\in[1.67,6.53]$M, while at the approximately
\$619M full-hardening threshold both objectives become zero. Hence, resilience
benefits are strongly front-loaded, and complete hardening ultimately
eliminates both residual losses and the conflict between objectives.

\subsection{Operational and distributional outcomes}
\label{sec:operational_distributional}

The Pareto trade-offs correspond to substantial changes in evacuation and
hardening decisions. As shown in Table~\ref{tbl:representative_outcomes},
expected evacuation demand falls from about 350 patients at the \$50M budget
to fewer than 13 patients at \$600M. Maximum scenario demand declines even
more sharply, from over 6{,}700 patients at \$50M to fewer than 200 at \$600M.
Although median evacuation demand is zero for nearly all representative
policies, these maxima reveal substantial tail risk at low and intermediate
budgets. The dummy receiver is never used because available physical receiving
capacity is sufficient to accommodate all evacuation demand.

Second, hardening investments expand rapidly as the available budget increases. At \$50M, only about 9--11\% of exposed facilities receive any protection, and the system-wide average protection depth remains below 5\% of the maximum feasible level. At \$300M, hardening coverage increases to approximately 43--48\%, while average protection depth reaches about 35--40\%. At \$600M, nearly all exposed facilities (93--96\%) receive some hardening, and average protection depth reaches approximately 91--95\% of the maximum feasible level. These patterns are consistent with a breadth-before-depth investment strategy: under limited budgets, the model favors distributing partial protection across facilities against lower-depth flood realizations rather than fully hardening a small number of facilities. As additional funding becomes available, protection levels deepen, extending coverage to higher-depth, less frequently realized flood conditions.

Facility capacity remains the common driver of prioritization across
the frontier: restoration loss, service-disruption impact, and evacuation demand increase with bed capacity. Tightening the disruption-impact constraint nevertheless changes the allocation of protection by placing greater weight on facilities located in higher-SVI census tracts. Figure~\ref{fig:budget_share_patterns}
shows that this reallocation is most pronounced at low-to-moderate budgets, where the budget is binding but still permits meaningful substitution across facilities.
\begin{figure}[!t]
\centering
\setlength{\tabcolsep}{4pt}
\begin{tabular}{cccc}
& \textbf{\$50M} & \textbf{\$300M} & \textbf{\$600M} \\

\rotatebox{90}{Budget share (\%)}
&
\includegraphics[width=0.30\textwidth]{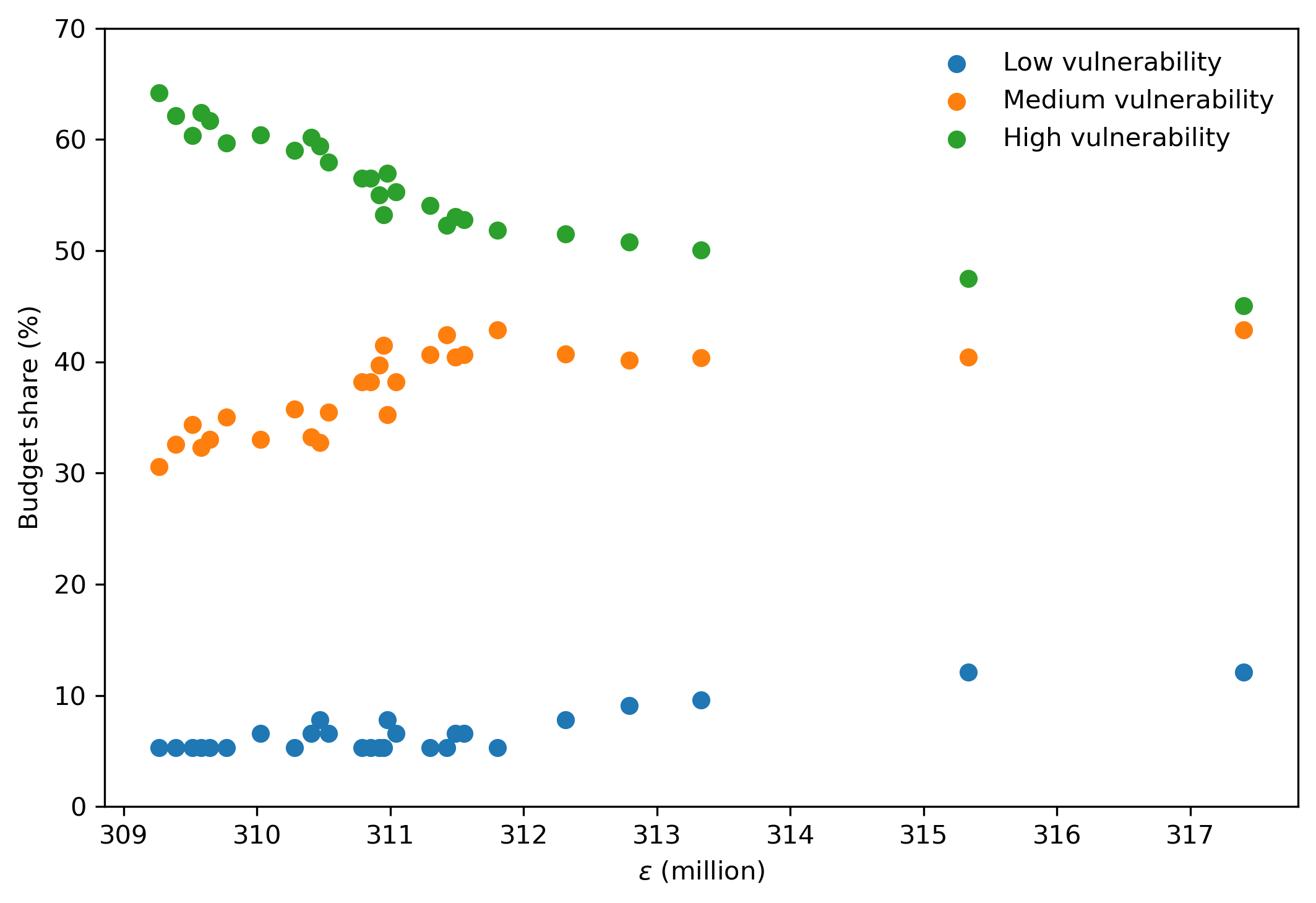}
&
\includegraphics[width=0.30\textwidth]{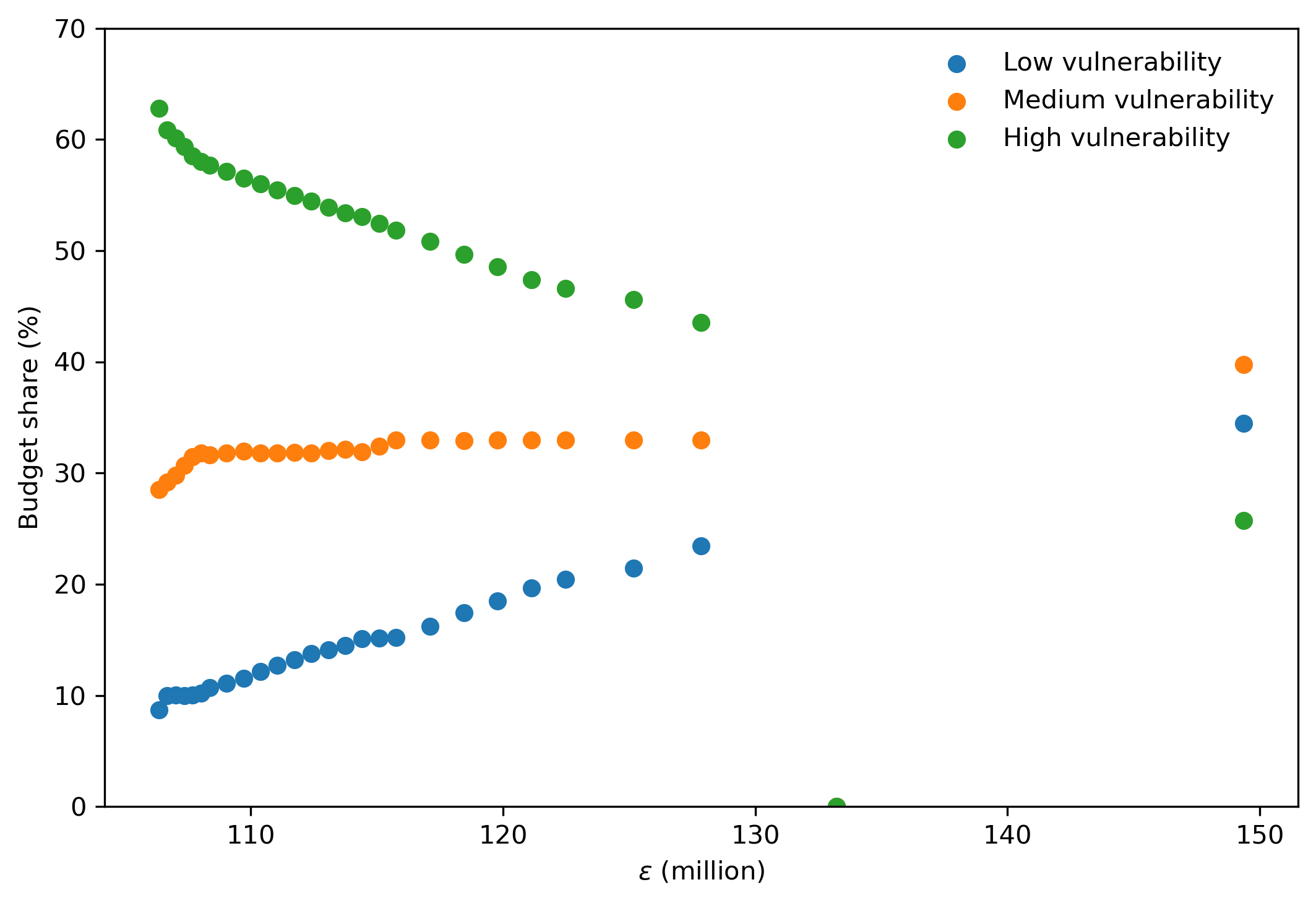}
&
\includegraphics[width=0.30\textwidth]{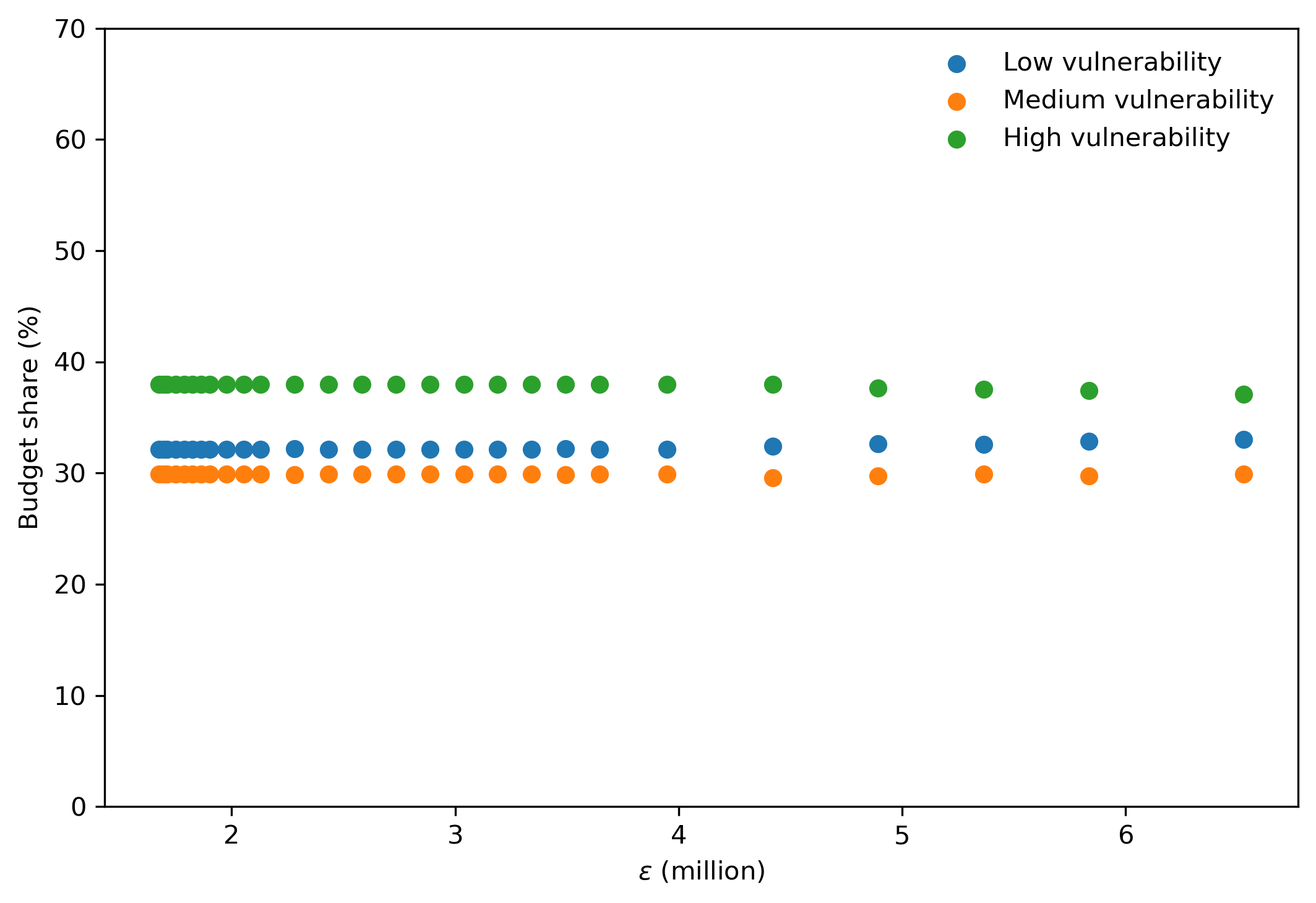}
\end{tabular}

\caption{Share of the hardening budget allocated to facilities located in
low-, medium-, and high-SVI census tracts. Moving from right to left
within each panel corresponds to a tighter service-disruption impact constraint and greater emphasis on reducing $f_2$.}
\label{fig:budget_share_patterns}
\end{figure}

Under scarce budgets, protection is concentrated more deeply at hospitals. As the budget expands, protection extends to more nursing homes, and impact-focused policies redirect part of this investment toward nursing
homes in high-vulnerability communities. Near full hardening, vulnerability- and facility-type differences largely disappear because almost all exposed
facilities can be protected. This sequencing is especially visible at the
\$300M knee: 78.26\% of exposed hospitals, but only 35.35\% of exposed nursing
homes, receive protection. Nevertheless, nursing homes absorb 58.15\% of the
budget because they constitute 99 of the 122 exposed facilities. Detailed
results by vulnerability group and facility type are provided in the
Supplementary Material.

\subsection{Value of stochastic planning}
\label{sec:VSS}

We compare the stochastic recourse problem with a deterministic
expected-value (EV) approximation in which scenario-specific flood depths are
replaced by their probability-weighted averages. Let $y^{EV}$ denote the resulting
hardening policy, $z_{f_1}^{EEV}(B,\varepsilon)$ its expected economic cost
when evaluated over the original scenario set, and
$z_{f_1}^{RP}(B,\varepsilon)$ the optimal value of the stochastic recourse problem. The value of the stochastic solution is
$\mathrm{VSS}(B,\varepsilon)
=
z_{f_1}^{EEV}(B,\varepsilon)
-
z_{f_1}^{RP}(B,\varepsilon)$.

 As shown in
Table~\ref{tbl:stochastic_value}, the resulting VSS increases from \$209.83M
at the \$50M budget to \$755.30M at \$600M. At \$50M, the comparison is made
at a common binding $\varepsilon$ threshold. At \$300M and \$600M, the EV-derived threshold lies above the entire stochastic Pareto frontier because the EV policy cannot attain the lower $f_2$ values achieved by the stochastic solutions. The impact constraint is
therefore nonbinding, and the stochastic benchmark reduces to the cost-minimizing solution.

\begin{table}[H]
\centering
\small
\setlength{\tabcolsep}{5pt}
\caption{Value of stochastic planning and perfect information. Monetary
values are reported in \$M. VSS is evaluated at
$\varepsilon=314.75$M, and EVPI is evaluated at the cost-minimizing endpoint.}
\label{tbl:stochastic_value}
\begin{tabular}{crrrr}
\toprule
Budget
& VSS
& VSS (\%)
& EVPI
& EVPI (\%) \\
\midrule
\$50M  & 209.83 & 27.41 & 121.48 & 21.89 \\
\$300M & 563.63 & 73.62 &  21.74 & 10.77 \\
\$600M & 755.30 & 98.65 &  10.31 & 100.00 \\
\bottomrule
\end{tabular}
\end{table}

We also compute the expected value of perfect information (EVPI) at the
cost-minimizing endpoint:
\[
\mathrm{EVPI}_{f_1}(B)
=
z_{f_1}^{RP}(B)-z_{f_1}^{WS}(B),
\]
where $z_{f_1}^{WS}(B)$ is the probability-weighted cost obtained by optimizing
separately for each realized scenario. Perfect information reduces expected
cost by \$121.48M at \$50M and \$21.74M at \$300M. At \$600M, it eliminates
the remaining \$10.31M in expected cost. Thus, scenario-based planning avoids
the substantial underinvestment produced by average-depth planning, while
perfect foresight provides a separate additional benefit.

\subsection{Algorithmic performance across scenario sets}
\label{sec:algorithmic_performance}

Table \ref{tbl:algorithm_summary} compares total runtime of extensive, Benders, and the LD methods over the 25 common interior $\varepsilon$-points.

\begin{table}[!t]
\centering
\small
\setlength{\tabcolsep}{5pt}
\caption{Total runtime in minutes for the 25 interior $\varepsilon$-points.
OOM$(n)$ indicates that $n$ points terminated because of insufficient memory.}
\label{tbl:algorithm_summary}
\begin{tabular}{llrrr}
\toprule
Scenario set & Budget & Extensive & Benders & LD \\
\midrule
\multirow{3}{*}{$\Smain$}
& \$50M  & 235       & 26  & 11 \\
& \$300M & OOM (2)   & 55  & 12 \\
& \$600M & OOM (2)   & 28  & 15 \\
\midrule
\multirow{3}{*}{$\Spercentile$ ($|\mathcal{S}|=81$)}
& \$50M  & OOM (25) & 83  & 28 \\
& \$300M & OOM (25) & 142 & 25 \\
& \$600M & OOM (25) & 107 & 33 \\
\bottomrule
\end{tabular}
\end{table}

The extensive approach is the most computationally demanding and becomes
memory-limited as the instance size grows. Benders solves every instance
exactly, requiring 26--55 minutes under $\Smain$ and 83--142 minutes under
$\Spercentile$. LD is consistently fastest, completing the corresponding
frontiers in 11--15 and 25--33 minutes, respectively. Thus, tripling the
scenario count increases runtime for both decomposition methods, but they
remain tractable, whereas the extensive formulation fails at every
$\Spercentile$ point.

Computational difficulty is not uniform across the frontier. For the exact
methods, the \$300M instances are generally the most demanding, particularly
near tight impact thresholds and the knee region, where small changes in
$\varepsilon$ can induce different hardening configurations. Complete per-point runtime
statistics are provided in the Supplementary Material.

Table~\ref{tbl:ld_gap_performance} summarizes the relative duality gaps
obtained by LD over the 25 interior $\varepsilon$-points. At the \$50M and
\$300M budgets, LD performs consistently well across both scenario sets:
mean gaps range from $0.25\%$ to $0.34\%$, and maximum gaps remain below
$1.5\%$.

\begin{table}[!t]
\centering
\small
\setlength{\tabcolsep}{6pt}
\renewcommand{\arraystretch}{1.1}
\caption{Relative duality-gap statistics for the Lagrangian-dual method}
\label{tbl:ld_gap_performance}
\begin{tabular}{cccc}
\toprule
\textbf{Budget}
& \begin{tabular}{c}Mean gap\\(\%)\end{tabular}
& \begin{tabular}{c}Median gap\\(\%)\end{tabular}
& \begin{tabular}{c}Maximum gap\\(\%)\end{tabular} \\
\midrule

\multicolumn{4}{c}{
\textit{$\mathcal{S}^{\text{main}}$ ($|\mathcal{S}|=27$)}} \\
\midrule
\$50M  & 0.34 & 0.35 & 0.61 \\
\$300M & 0.27 & 0.24 & 0.81 \\
\$600M & 2.31 & 0.99 & 6.74 \\

\midrule
\multicolumn{4}{c}{
\textit{$\mathcal{S}^{\text{perc}}$ ($|\mathcal{S}|=81$)}} \\
\midrule
\$50M  & 0.34 & 0.30 & 1.18 \\
\$300M & 0.25 & 0.13 & 1.47 \\
\$600M & 7.31 & 5.49 & 19.28 \\

\bottomrule
\end{tabular}

\vspace{1mm}
\begin{minipage}{0.95\textwidth}
\footnotesize
Notes: Statistics are computed over the 25 interior $\varepsilon$-points.
For each point, LD runs for at most 1000 iterations. The reported relative
duality gap is
$\frac{UB-LB}{\max\{1,UB\}}$, where $LB$ is the Lagrangian-dual
bound and $UB$ is the objective value of the best recovered feasible solution.
\end{minipage}
\end{table}

The reported gaps increase at the \$600M budget. At this budget, the attainable $f_2$ range is
narrow, and portions of the frontier are nearly flat. The primal-recovery
procedure may consequently retain the same feasible solution for several
neighboring $\varepsilon$-values. Thus, a
comparatively large reported bound gap does not necessarily imply that the
recovered feasible solution is equally far from the exact Pareto frontier.

To distinguish bound quality from primal-solution quality,
Table~\ref{tbl:ld_quality} compares representative recovered LD solutions with
the corresponding exact Pareto solutions. The deviations are defined as
\[
\mathrm{Dev}(f_i)
=
\frac{f_i^{\mathrm{feas}}-f_i^*}{f_i^*},
\qquad i\in\{1,2\}.
\]
The $f_2$ deviation is signed: a positive value indicates greater disruption
impact than the exact solution, whereas a negative value indicates that the
recovered solution over-satisfies the $\varepsilon$-constraint.

\begin{table}[!t]
\centering
\small
\setlength{\tabcolsep}{3pt}
\renewcommand{\arraystretch}{1.1}
\caption{Quality of dual bound and primal-recovery solutions relative to the exact Pareto frontier for representative interior $\varepsilon$-points.}
\label{tbl:ld_quality}
\begin{tabular}{lccccccccc}
\toprule
Budget 
& Policy
& $LB$
& $UB$
& $f_2^{\text{feas}}$
& $f_1^{*}$
& $f_2^{*}$
& $f_1$ Dev. (\%)
& $f_2$ Dev. (\%)
& Gap (\%) \\
\midrule

\multicolumn{10}{c}{\textit{$\mathcal{S}^{\text{main}}$ ($|\mathcal{S}|=27$)}} \\
\midrule

\multirow{3}{*}{\$50M}
& First interior & 568.21 & 568.99 & 309.40 & 568.65 & 309.38 & 0.06 & 0.01 & 0.14 \\
& Knee           & 557.80 & 559.44 & 311.69 & 559.05 & 311.70 & 0.07 & 0.00 & 0.29 \\
& Last interior  & 555.38 & 556.16 & 315.01 & 555.54 & 315.16 & 0.11 & -0.05 & 0.14 \\

\midrule
\multirow{3}{*}{\$300M}
& First interior & 234.30 & 235.86 & 106.49 & 234.50 & 106.68 & 0.58 & -0.18 & 0.66 \\
& Knee           & 211.11 & 211.84 & 119.76 & 211.19 & 119.78 & 0.31 & -0.01 & 0.34 \\
& Last interior  & 203.45 & 204.34 & 133.13 & 203.58 & 133.23 & 0.37 & -0.08 & 0.44 \\

\midrule
\multirow{3}{*}{\$600M}
& First interior & 13.02 & 13.64 & 1.69 & 13.26 & 1.69 & 2.68 & 0.00 & 4.53 \\
& Knee           & 10.28 & 10.34 & 3.33 & 10.31 & 3.34 & 0.27 & -0.35 & 0.51 \\
& Last interior  & 10.29 & 10.34 & 3.33 & 10.31 & 5.79 & 0.27 & -42.53 & 0.46 \\

\midrule
\multicolumn{10}{c}{\textit{$\mathcal{S}^{\text{perc}}$ ($|\mathcal{S}|=81$)}} \\
\midrule

\multirow{3}{*}{\$50M}
& First interior & 592.87 & 596.71 & 324.09 & 596.43 & 323.94 & 0.05 & 0.05 & 0.64 \\
& Knee           & 583.62 & 586.36 & 325.91 & 584.76 & 326.03 & 0.27 & -0.04 & 0.47 \\
& Last interior  & 580.87 & 582.15 & 327.73 & 581.25 & 330.19 & 0.15 & -0.74 & 0.22 \\

\midrule
\multirow{3}{*}{\$300M}
& First interior & 273.20 & 274.85 & 133.98 & 273.42 & 134.03 & 0.53 & -0.04 & 0.60 \\
& Knee           & 249.18 & 250.79 & 144.40 & 249.34 & 144.41 & 0.58 & -0.01 & 0.64 \\
& Last interior  & 241.89 & 242.20 & 159.43 & 241.94 & 162.09 & 0.11 & -1.64 & 0.13 \\

\midrule
\multirow{3}{*}{\$600M}
& First interior & 100.88 & 109.00 & 39.38 & 106.14 & 39.38 & 2.69 & 0.00 & 7.46 \\
& Knee           & 87.27  & 88.05  & 40.25 & 87.79  & 40.25 & 0.30 & 0.00 & 0.89 \\
& Last interior  & 75.15  & 75.50  & 44.61 & 75.18  & 44.61 & 0.42 & 0.00 & 0.46 \\

\bottomrule
\end{tabular}

\vspace{2mm}
\begin{minipage}{\textwidth}
\footnotesize
Notes: $LB$ denotes the Lagrangian dual bound on the primary objective $f_1$ and
$UB=f_1^{\text{feas}}$ denotes the objective value of the recovered feasible solution.
Values are reported in millions.
\end{minipage}
\end{table}

Direct comparison with the exact frontier confirms that the reported duality
gap can be more conservative than the actual deviation of the recovered
solution. Across both scenario sets and all three budgets, the recovered LD
knee solutions deviate from the exact $f_1$ values by at most $0.58\%$ and
from the exact $f_2$ values by at most $0.35\%$ in absolute terms. Even at the
more difficult first-interior points for the \$600M budget, the reported gaps
are $4.53\%$ and $7.46\%$, whereas the corresponding $f_1$ deviations are only
$2.68\%$ and $2.69\%$, respectively, and the $f_2$ values match the exact
solutions to the reported precision.

The large negative $f_2$ deviation of $-42.53\%$ at the last interior point
under $\mathcal{S}^{\text{main}}$ and the \$600M budget does not indicate poor
feasibility. In this nearly flat region of the frontier, a negligible increase
in $f_1$ can accompany a substantial reduction in $f_2$. The recovered
solution has an $f_1$ value only $0.27\%$ above that of the corresponding exact
Pareto point while achieving substantially lower disruption impact. The
negative deviation therefore reflects over-satisfaction of the impact
constraint rather than poor primal-recovery performance. Taken together, these comparisons show that LD recovers high-quality feasible
solutions at representative points across the frontier, with particularly
close agreement in the policy-relevant knee region, even when its reported
duality gap provides a conservative assessment of primal-solution quality.

Benders is therefore the preferred method when exact optimality is required, whereas LD provides a substantially faster approximation for large instances and repeated analyses.

%% file: 06_conclusion.tex
\section{Conclusion}
\label{sec:conclusion}

This study develops a bi-objective two-stage stochastic optimization framework
for facility resilience planning under uncertain flood exposure. The model
jointly determines proactive hardening investments and scenario-dependent
evacuation decisions while balancing expected economic loss against service disruption impact. An exact multi-cut Benders decomposition exploits the transportation
structure of the evacuation-flow subproblems after retaining the discrete
recourse decisions in the master, while a Lagrangian-dual method with
tailored primal recovery provides a scalable approximation for larger
instances. An adaptive $\varepsilon$-constraint
procedure supports systematic exploration of the resulting Pareto frontier.

The Texas healthcare case study shows that resilience planning based only on
economic loss can produce materially different investment portfolios from
planning that also limits care disruption impact. Relatively small movements away
from the economic-loss-minimizing endpoint can yield substantial reductions in disruption impact. In this case study, facility capacities remain central drivers of
prioritization, while tighter impact limits shift protection toward
facilities in more socially vulnerable communities. The analysis also
shows that mitigation benefits are strongly front-loaded and that
deterministic average flood depth planning substantially understates the protection needed given the modeled scenario variability.

The computational results establish the practical value of decomposition. Solving the model with a commercial solver becomes memory-limited as the number of scenarios
and evacuation arcs grows, whereas Benders solves all tested instances exactly. The Lagrangian-dual method is substantially faster and recovers solutions close to the exact frontier, making it
well suited to large-scale analysis and repeated planning studies.
Benders remains preferable when strict optimality guarantees are required.

Several limitations suggest natural extensions. First, the model assumes a
single decision-maker with system-wide authority and therefore abstracts from
the strategic interactions, heterogeneous incentives, and cost-sharing
arrangements that arise among public agencies, healthcare providers, insurers,
and facility operators. A game-theoretic multi-agent formulation could capture
these decentralized decisions and the resulting coordination challenges.
Second, the analysis assumes no pre-existing protection and represents flood
exposure beyond the installed protection level as complete facility shutdown.
Future models could incorporate existing defenses, partial operability, and
phased restoration. Further extensions could include dynamic investment, infrastructure interdependencies, and endogenous recovery resources.

Although the empirical application focuses on healthcare resilience in Texas,
the framework applies more broadly to geographically distributed facility
systems in which disrupted sites generate pickup, transfer, or delivery
requirements and receiving sites have limited capacity. Examples include emergency shelters, humanitarian
logistics networks, warehouses, repair depots, public-service facilities, and other critical infrastructure systems. More generally, the study demonstrates
how multi-objective stochastic optimization can identify resilience strategies that are computationally tractable, economically efficient, and attentive to the distribution of service disruption impact across affected communities.

%% file: 7_supp.tex
\section{Additional methodological details}

\subsection{Lagrangian-dual method}
\label{sec:supp_ld_details}

This section provides pseudocodes for the Lagrangian-dual method described in the main manuscript.

\subsubsection{Complete primal-recovery routine}
\label{sec:supp_ld_primal}

Algorithm~\ref{alg:PrimalRecovery} summarizes the overall primal recovery routine, integrating \textsc{BudgetRepair}, \textsc{EpsilonRepair}, and \textsc{GreedyTransport}.

\begin{breakablealgorithm}
\caption{\textsc{BudgetRepair}$(y,B)$}
\label{alg:BudgetRepair}
\begin{algorithmic}[1]
\STATE \textbf{Input:} hardening plan $y$, budget $B$
\STATE \textbf{Output:} budget-feasible hardening plan $\hat y$

\STATE $\hat y \gets y$
\STATE $C \gets \sum_{j\in J^{\mathrm{send}}} C_j^H \hat y_j$
\IF{$C \le B$}
    \STATE \textbf{return} $\hat y$
\ENDIF

\STATE Initialize min-heap $\mathcal H_{\mathrm{rem}}$
\FOR{each $j\in J^{\mathrm{send}}$ with $\hat y_j\ge 1$}
    \STATE $\ell \gets \hat y_j$
    \STATE Compute removal score $\rho_j(\ell)$
    \STATE Push key $\rho_j(\ell)$ with index $(j,\ell)$
\ENDFOR

\WHILE{$C>B$}
    \STATE Extract $(j,\ell)$ with smallest key
    \IF{$\ell \neq \hat y_j$}
        \STATE \textbf{continue}
    \ENDIF
    \STATE $\hat y_j \gets \hat y_j - 1$
    \STATE $C \gets C - C_j^H$
    \IF{$\hat y_j \ge 1$}
        \STATE $\ell \gets \hat y_j$
        \STATE Compute updated removal score $\rho_j(\ell)$
        \STATE Push updated key into $\mathcal H_{\mathrm{rem}}$
    \ENDIF
\ENDWHILE

\STATE \textbf{return} $\hat y$
\end{algorithmic}
\end{breakablealgorithm}

\begin{breakablealgorithm}
\caption{\textsc{EpsilonRepair}$(y,B,\varepsilon)$}
\label{alg:EpsilonRepair}
\begin{algorithmic}[1]
\STATE \textbf{Input:} budget-feasible hardening plan $y$, budget $B$, threshold $\varepsilon$
\STATE \textbf{Output:} $\tilde y$ with $f_2(\tilde y)\le \varepsilon$, or \textsc{Fail}

\STATE $\tilde y \gets y$
\STATE $C \gets \sum_{j\in J^{\mathrm{send}}} C_j^H \tilde y_j$
\STATE $I \gets f_2(\tilde y)$
\IF{$I \le \varepsilon$}
    \STATE \textbf{return} $\tilde y$
\ENDIF

\STATE Initialize max-heap $\mathcal H_{\mathrm{add}}$
\FOR{each $j\in J^{\mathrm{send}}$ with $\tilde y_j < H_j^{\max}$}
    \STATE $\ell \gets \tilde y_j + 1$
    \STATE Compute addition score $\rho_j(\ell)$
    \STATE Push key $\rho_j(\ell)$ with index $(j,\ell)$
\ENDFOR

\WHILE{$I>\varepsilon$}
    \IF{$\mathcal H_{\mathrm{add}}$ empty}
        \STATE \textbf{return Fail}
    \ENDIF
    \STATE Extract $(j,\ell)$ with largest key
    \IF{$\ell \neq \tilde y_j+1$}
        \STATE \textbf{continue}
    \ENDIF
    \IF{$C + C_j^H > B$}
        \STATE \textbf{continue}
    \ENDIF
    \STATE $\tilde y_j \gets \tilde y_j + 1$
    \STATE $C \gets C + C_j^H$
    \STATE $I \gets I - \Delta_j^{\mathrm{imp}}(\tilde y_j)$
    \IF{$\tilde y_j < H_j^{\max}$}
        \STATE $\ell \gets \tilde y_j + 1$
        \STATE Compute updated addition score $\rho_j(\ell)$
        \STATE Push updated key into $\mathcal H_{\mathrm{add}}$
    \ENDIF
\ENDWHILE

\STATE \textbf{return} $\tilde y$
\end{algorithmic}
\end{breakablealgorithm}

\begin{breakablealgorithm}
\caption{\textsc{GreedyTransport}$(\bar\gamma)$}
\label{alg:GreedyTransport}
\begin{algorithmic}[1]
\STATE \textbf{Input:} evacuation indicators $\bar\gamma_{sj}$, demands $D_j$, capacities $A_k$, ordered candidate lists $K_{sj}$
\STATE \textbf{Output:} flows $\bar q$, or \textsc{Fail}

\FOR{each $s\in\mathcal S$}
    \STATE $\bar q_{sjk}\gets 0 \quad \forall (j,k)\in\tilde{\mathcal A}_s$
    \STATE $cap_k \gets A_k \quad \forall k\in\tilde J_s^{\mathrm{recv}}$
    \STATE $\mathcal J_s^{\mathrm{evac}} \gets \{j\in J^{\mathrm{send}}:\bar\gamma_{sj}=1\}$
    \STATE Sort \(\mathcal J_s^{\mathrm{evac}}\) by decreasing \(D_j\)
    \FOR{each $j\in\mathcal J_s^{\mathrm{evac}}$}
        \STATE $rem \gets D_j$
        \FOR{each $k\in K_{sj}$}
            \IF{$rem=0$}
                \STATE \textbf{break}
            \ENDIF
            \STATE $x \gets \min\{rem,cap_k\}$
            \STATE $\bar q_{sjk} \gets x$
            \STATE $cap_k \gets cap_k - x$
            \STATE $rem \gets rem - x$
        \ENDFOR
        \IF{$rem>0$}
            \STATE \textbf{return Fail}
        \ENDIF
    \ENDFOR
\ENDFOR
\STATE \textbf{return} $\bar q$
\end{algorithmic}
\end{breakablealgorithm}

\begin{breakablealgorithm}
\caption{\textsc{PrimalRecovery}$(y^t,B,\varepsilon)$}
\label{alg:PrimalRecovery}
\begin{algorithmic}[1]
\STATE \textbf{Input:} Lagrangian minimizer $y^t$, budget $B$, threshold $\varepsilon$
\STATE \textbf{Output:} feasible solution $(\bar y,\bar\gamma,\bar z,\bar q)$ and upper bound $UB^t$, or \textsc{Fail}

\STATE $\bar y \gets \textsc{BudgetRepair}(y^t,B)$ (Algorithm~\ref{alg:BudgetRepair})
\STATE $\bar y \gets \textsc{EpsilonRepair}(\bar y,B,\varepsilon)$ (Algorithm~\ref{alg:EpsilonRepair})
\IF{\textsc{EpsilonRepair} returns \textsc{Fail}}
    \STATE \textbf{return Fail}
\ENDIF

\FOR{each $(s,j)\in\mathcal S\times J^{\mathrm{send}}$}
    \STATE $\bar\gamma_{sj}\gets \mathbf{1}[F_{sj}>\bar y_j]$
    \STATE $\bar z_{sj}\gets \max\{F_{sj}-\bar y_j,0\}$
\ENDFOR

\STATE $\bar q \gets \textsc{GreedyTransport}(\bar\gamma)$ (Algorithm~\ref{alg:GreedyTransport})
\IF{\textsc{GreedyTransport} returns \textsc{Fail}}
    \STATE \textbf{return Fail}
\ENDIF

\STATE Compute
\[
UB^t
=
\sum_{s\in\mathcal S} P_s
\left(
\sum_{j\in J^{\mathrm{send}}} R_j \bar z_{sj}
+
\sum_{(j,k)\in\tilde{\mathcal A}_s} C^E_{jk}\bar q_{sjk}
\right)
\]
\STATE \textbf{return} $(\bar y,\bar\gamma,\bar z,\bar q,UB^t)$
\end{algorithmic}
\end{breakablealgorithm}

\subsubsection{Constructive heuristic for the first LD iteration}

\begin{breakablealgorithm}
\caption{\textsc{InitialFeasiblePoint}$(B,\varepsilon)$}
\label{alg:InitialFeasiblePoint}
\begin{algorithmic}[1]
\STATE \textbf{Input:} budget $B$, threshold $\varepsilon$, hardening costs $C_j^H$, limits $H_j^{\max}$, marginal gains $\Delta_j^{\mathrm{imp}}(\ell)$ and $\Delta_j^{\mathrm{rest}}(\ell)$
\STATE \textbf{Output:} feasible hardening plan $y$, or \textsc{Fail}

\STATE $y_j \gets 0 \quad \forall j\in J^{\mathrm{send}}$
\STATE $C \gets 0 = \sum_{j\in J^{\mathrm{send}}} C_j^H y_j$
\STATE $I \gets f_2(y)=\sum_{s\in\mathcal S} P_s \sum_{j\in J^{\mathrm{send}}} V_j(F_{sj}-y_j)_+ $
\IF{$I \le \varepsilon$}
    \STATE \textbf{return} $y$
\ENDIF

\STATE Initialize max-heap $\mathcal H_{\mathrm{add}}$
\FOR{each $j\in J^{\mathrm{send}}$ with $H_j^{\max}\ge 1$}
    \STATE $\ell \gets 1$
    \STATE Push key $\big(\Delta_j^{\mathrm{imp}}(\ell)/C_j^H,\; \Delta_j^{\mathrm{rest}}(\ell)/C_j^H\big)$ with index $(j,\ell)$
\ENDFOR

\WHILE{$I>\varepsilon$}
    \IF{$\mathcal H_{\mathrm{add}}$ empty}
        \STATE \textbf{return Fail}
    \ENDIF
    \STATE Extract facility-level pair $(j,\ell)$ with largest key
    \IF{$\ell \neq y_j+1$}
        \STATE \textbf{continue}
    \ENDIF
    \IF{$C + C_j^H > B$}
        \STATE \textbf{continue}
    \ENDIF
    \STATE $y_j \gets y_j + 1$
    \STATE $C \gets C + C_j^H$
    \STATE $I \gets I - \Delta_j^{\mathrm{imp}}(y_j)$
    \IF{$y_j < H_j^{\max}$}
        \STATE $\ell \gets y_j + 1$
        \STATE Push updated key for $(j,\ell)$ into $\mathcal H_{\mathrm{add}}$
    \ENDIF
\ENDWHILE

\STATE \textbf{return} $y$
\end{algorithmic}
\end{breakablealgorithm}

\subsubsection{Complete LD algorithm}

\begin{breakablealgorithm}
\caption{Lagrangian Dual Frontier Sweep with Warm Starts and Primal Recovery}
\label{alg:LD_polyak_eps}
\begin{algorithmic}[1]
\STATE \textbf{Input:} model parameters, $B$, $\mathrm{gap\_tol}$, ordered frontier values $\varepsilon_1<\cdots<\varepsilon_{\mathcal{M}}$
\STATE \textbf{Output:} for each $\varepsilon_{\hat{m}}$, lower bound $LB_{\hat{m}}$, upper bound $UB_{\hat{m}}$, and incumbent solution

\STATE Precompute restricted sets $K_{sj}$, $\tilde J_s^{\mathrm{recv}}$, $\tilde{\mathcal A}_s$, and scenario orderings
\STATE Initialize warm-start data $(\mu^{\mathrm{ws}},\pi^{\mathrm{ws}},\nu^{\mathrm{ws}},y^{\mathrm{feas}})\gets (0,0,0,\varnothing)$

\FOR{$\hat{m}=1,\dots,M$}
    \vspace{-2pt}
    \STATE $\varepsilon \gets \varepsilon_{\hat{m}}$
    \STATE $LB\gets -\infty$, \quad $UB\gets +\infty$

    \IF{$\hat{m}=1$ and $y^{\mathrm{feas}}=\varnothing$}
        \STATE Initialize $(\mu^0,\pi^0,\nu^0)\gets (0,0,0)$
        \STATE Run \textsc{InitialFeasiblePoint}$(B,\varepsilon)$ 
        \IF{successful}
            \STATE Evaluate its objective and update $UB$ and set incumbent feasible solution
        \ENDIF
        \vspace{-3pt}
    \ELSE
        \STATE Warm start multipliers $(\mu^0,\pi^0,\nu^0)\gets (\mu^{\mathrm{ws}},\pi^{\mathrm{ws}},\nu^{\mathrm{ws}})$
        \IF{$y^{\mathrm{feas}}\neq\varnothing$}
            \STATE Initialize incumbent using $y^{\mathrm{feas}}$
        \ENDIF
        \vspace{-3pt}
    \ENDIF

    \FOR{$t=0,\dots,T-1$}
        \STATE Apply \textsc{EnumerateHardening}$(\mu^t,\pi^t,\nu^t)$ and obtain $(y^t,\gamma^t,z^t,q^t)$ and $L^t$
        \STATE $LB\gets \max\{LB,L^t\}$
        \STATE Apply \textsc{PrimalRecovery}$(y^t,B,\varepsilon)$ 
        \IF{successful and $UB^t<UB$}
            \STATE Update incumbent and set $UB\gets UB^t$
        \ENDIF
        \IF{$\dfrac{UB-LB}{\max\{1,\,UB\}} \le \mathrm{gap\_tol}$}
            \STATE \textbf{break}
        \ENDIF
        \STATE Compute subgradients $(g_\mu^t,g_\pi^t,g_\nu^t)$ and step size $\eta^t$
        \STATE Update $\mu^{t+1}$, $\pi^{t+1}$, and $\nu^{t+1}$
    \ENDFOR

    \STATE Store $(LB,UB)$ and the incumbent solution for $\varepsilon_{\hat{m}}$
    \STATE Set warm-start data $(\mu^{\mathrm{ws}},\pi^{\mathrm{ws}},\nu^{\mathrm{ws}})\gets (\mu^t,\pi^t,\nu^t)$
    \STATE Update $y^{\mathrm{feas}}$ with the best feasible hardening plan found for $\varepsilon_{\hat{m}}$, if any
\ENDFOR

\STATE \textbf{return} all $(LB_{\hat{m}},UB_{\hat{m}})$ values and incumbent solutions along the frontier
\end{algorithmic}
\end{breakablealgorithm}

\section{Case-study data and parameter construction}
\label{sec:supp_case_study}

\subsubsection{Historical tropical-cyclone events}
\label{sec:supp_tc_events}

Table~\ref{tbl:supp_tc_events} lists the 27 historical tropical cyclones used to construct the scenario sets.

\begin{table}[H]
\centering
\caption{Historical tropical-cyclone events included in the scenario sets
\citep{Tabassum2024Dissertation}}
\label{tbl:supp_tc_events}
\small
\begin{tabular}{lll}
\toprule
Storm name & Maximum intensity & Date range\\
\midrule
Harvey & Category 4 Hurricane & Aug. 24, 2017--Sep. 02, 2017\\
Cindy & Tropical Storm & Jun. 19, 2017--Jun. 28, 2017\\
Bill & Tropical Storm & Jun. 16, 2015--Jun. 25, 2015\\
Isaac & Category 1 Hurricane & Aug. 27, 2012--Sep. 05, 2012\\
Don & Tropical Storm & Jul. 27, 2011--Aug. 08, 2011\\
Hermine & Tropical Storm & Sep. 04, 2010--Sep. 13, 2010\\
Two & Tropical Depression & Jul. 07, 2010--Jul. 16, 2010\\
Ike & Category 4 Hurricane & Sep. 10, 2008--Sep. 19, 2008\\
Gustav & Category 4 Hurricane & Sep. 01, 2008--Sep. 10, 2008\\
Edouard & Tropical Storm & Aug. 03, 2008--Aug. 12, 2008\\
Dolly & Category 2 Hurricane & Jul. 20, 2008--Jul. 29, 2008\\
Humberto & Category 1 Hurricane & Sep. 12, 2007--Sep. 21, 2007\\
Erin & Tropical Storm & Aug. 15, 2007--Aug. 24, 2007\\
Rita & Category 5 Hurricane & Sep. 18, 2005--Sep. 27, 2005\\
Ivan & Category 5 Hurricane & Sep. 20, 2004--Sep. 29, 2004\\
Grace & Tropical Storm & Aug. 30, 2003--Sep. 08, 2003\\
Erika & Category 1 Hurricane & Aug. 14, 2003--Aug. 23, 2003\\
Claudette & Category 1 Hurricane & Jul. 08, 2003--Jul. 17, 2003\\
Fay & Tropical Storm & Sep. 05, 2002--Sep. 14, 2002\\
Bertha & Tropical Storm & Aug. 04, 2002--Aug. 13, 2002\\
Allison & Tropical Storm & Jun. 05, 2001--Jun. 14, 2001\\
Unnamed & Tropical Depression & Sep. 08, 2000--Sep. 17, 2000\\
Frances & Tropical Storm & Sep. 08, 1998--Sep. 17, 1998\\
Charley & Tropical Storm & Aug. 21, 1998--Aug. 30, 1998\\
Dean & Tropical Storm & Jul. 28, 1995--Aug. 06, 1995\\
Lidia & Category 4 Hurricane & Sep. 08, 1993--Sep. 17, 1993\\
Arlene & Tropical Storm & Jun. 18, 1993--Jun. 27, 1993\\
\bottomrule
\end{tabular}
\end{table}

\subsection{Parameter construction details}
\label{sec:supp_parameter_construction}

This section documents the intermediate assumptions and calculations underlying the parameter values reported in the main manuscript.

\subsubsection{Facility geometry and hardening cost}
\label{sec:supp_hardening_cost}

Facility-specific perimeter and area data are unavailable. We therefore
estimate the effective perimeter $X_j$ and area $\varphi_j$ from bed capacity
using building-area statistics reported by \citet{definitivehc_sqft}.
For each bed-capacity category, the implied perimeter is computed assuming a
square footprint,
\[
P=4\sqrt{A},
\]
which provides a lower-bound approximation among rectangular building
footprints.

\begin{table}[!t]
\centering
\caption{Facility-area and implied-perimeter inputs by bed-capacity category}
\label{tab:supp_perimeter_from_area}
\small
\begin{tabular}{lccc}
\toprule
Bed-capacity bin
& Mean area (ft$^2$)
& Implied perimeter (ft)
& Perimeter-to-bed ratio \\
\midrule
$<25$        & 78,962      & 1,124 & 90.0 \\
$25$--$49$   & 111,779     & 1,337 & 36.2 \\
$50$--$99$   & 183,328     & 1,713 & 22.8 \\
$100$--$249$ & 430,060     & 2,623 & 15.0 \\
$250$--$499$ & 909,558     & 3,815 & 10.2 \\
$500$--$999$ & 1,868,317   & 5,468 & 7.3 \\
$1,000+$     & 3,264,601   & 7,227 & 7.2 \\
\bottomrule
\end{tabular}
\end{table}

Let $0=b_0<b_1<\cdots<b_K$ denote the bed-capacity breakpoints and let
$r_k$ be the perimeter-to-bed ratio for category $k$. We estimate
\[
X_j
=
\sum_{k=1}^{K}
r_k
\max\{\min\{O_j,b_k\}-b_{k-1},0\}.
\]
Facility area $\varphi_j$ is calculated analogously using the corresponding
area-to-bed ratios. These piecewise-linear mappings are increasing and
concave, representing diminishing footprint growth as facility size
increases.

The hardening coefficient is $C_j^H=uX_j$. FEMA planning guidance reports a
wide range of floodwall construction costs
\citep{SCCOG_critical_facilities}; we use
$u=\$200$ per linear foot per foot of wall height.

As an order-of-magnitude check, the Lourdes Hospital floodwall project cost
approximately \$9.45 million in 2024 dollars for an estimated perimeter of
5,750 ft and a maximum height of 11 ft
\citep{fema_advanced_2012}, implying
\[
u_{\mathrm{Lourdes}}
=
\frac{9{,}450{,}000}{5{,}750\times11}
\approx \$149.
\]

\subsubsection{Restoration and business interruption coefficients}
\label{sec:supp_restoration_cost}

Physical-damage coefficients are derived from the most-likely U.S. Army
Corps of Engineers depth--damage curves for hospitals and nursing homes
\citep{naccs_depth_damage}. Because facility-specific finished-floor
elevations are unavailable, we assume that site flood depth represents
interior inundation depth.

For facility type $t$ and damage component
$c\in\{\mathrm{str},\mathrm{con}\}$, where $\mathrm{str}$ and $\mathrm{con}$ denote structure and contents damage, respectively, the depth--damage observations
$d_c^t(z)$ at positive inundation depths are approximated by a line through
the origin:
\[
d_c^t(z)\approx\beta_c^t z,
\qquad
\beta_c^t
=
\frac{\sum_z z\,d_c^t(z)}
{\sum_z z^2}.
\]
The fitted marginal damage fractions are
\[
\beta_{\mathrm{str}}^H=0.0566,
\qquad
\beta_{\mathrm{con}}^H=0.0704,
\]
and
\[
\beta_{\mathrm{str}}^N=0.0618,
\qquad
\beta_{\mathrm{con}}^N=0.0859.
\]

Using Hazus structure and contents replacement values
\citep{hazus_inventory_tm},
\[
r_{\mathrm{phys}}^t
=
\beta_{\mathrm{str}}^t v_{\mathrm{str}}^t
+
\beta_{\mathrm{con}}^t v_{\mathrm{con}}^t,
\]
which gives
\[
r_{\mathrm{phys}}^H=\$63.7,
\qquad
r_{\mathrm{phys}}^N=\$27.4
\quad
(\text{USD/ft}^2/\text{ft}).
\]

Following the Hazus business-interruption assumptions
\citep{hazus_flood_tm,hazus_inventory_tm}, the effective daily loss rate is
\[
\pi_{\mathrm{BI}}
=
(1-0.60)(0.221+0.519)
=
\$0.296
\quad
(\text{USD/ft}^2/\text{day}).
\]
The Hazus stepwise restoration-time data are approximated by
\[
RT^t(z)\approx\alpha^t z,
\]
using an origin-constrained least-squares fit over depths 0--24 ft. This gives
\[
\alpha^H=53.4,
\qquad
\alpha^N=51.3
\quad
(\text{days/ft}).
\]
Therefore,
\[
r_{\mathrm{BI}}^t=\pi_{\mathrm{BI}}\alpha^t,
\]
yielding
\[
r_{\mathrm{BI}}^H=\$15.8,
\qquad
r_{\mathrm{BI}}^N=\$15.2
\quad
(\text{USD/ft}^2/\text{ft}).
\]
Combining the two components gives
\[
R_j
=
\begin{cases}
79.5\,\varphi_j, & t(j)=H,\\
42.6\,\varphi_j, & t(j)=N.
\end{cases}
\]

\subsubsection{Evacuation-cost construction}
\label{sec:supp_evacuation_cost}

We assume one ambulance trip per patient and a two-person crew. The 2019 City
of Houston EMS mileage charge of \$15 per mile is converted to 2024 dollars
using an inflation factor of 1.23:
\[
c_{\mathrm{mile}}
=
15(1.23)
=
\$18.45\text{/mile}.
\]
Using a 2019 mean hourly wage of \$18.67 for emergency medical technicians
and paramedics, and assuming one hour of crew time per patient transfer,
\[
c_{\mathrm{labor}}
=
2(18.67)(1.23)
\approx \$45.9.
\]
The resulting aggregate evacuation cost is
\[
C_{jk}^E
=
c_{\mathrm{mile}}d_{jk}
+
c_{\mathrm{labor}}
=
18.45d_{jk}+45.9.
\]
The one-hour transfer assumption is consistent with related
disaster-evacuation studies \citep{KIM2024104518}.

\section{Supplementary computational and policy results}

\subsection{Budget-dependent Pareto endpoints}

Table~\ref{tbl:payoff_endpoints_by_budget} reports the Pareto endpoint values
and the break-even measure
\[
N(B)=\frac{B}{f_1(0)-f_1(B)},
\]
where $N(B)$ denotes the number of comparable events required for cumulative
expected economic-loss reductions to recover the hardening investment.

\begin{table}[!t]
\centering
\caption{Pareto extreme-point values and break-even event counts by budget.
All objective values are reported in $10^{6}$.}
\label{tbl:payoff_endpoints_by_budget}
\setlength{\tabcolsep}{4pt}
\renewcommand{\arraystretch}{1.05}
\begin{tabular}{ccccc|cc}
\toprule
\textbf{Budget}
& \multicolumn{2}{c}{$f_1$ endpoints (\$)}
& \multicolumn{2}{c}{$f_2$ endpoints}
& \multicolumn{2}{c}{$N(B)=\dfrac{B}{f_1(0)-f_1(B)}$} \\
\cmidrule(lr){2-3}
\cmidrule(lr){4-5}
\cmidrule(lr){6-7}
\textbf{\$}
& $f_1^{\min}$
& $f_1^{\max}$
& $f_2^{\min}$
& $f_2^{\max}$
& at $f_1^{\min}$
& at $f_1^{\max}$ \\
\midrule
0   & 1032.13 & 1032.13 & 607.62 & 607.62 & --    & --    \\
50  & 555.27  & 569.42  & 309.26 & 317.40 & 0.105 & 0.108 \\
100 & 431.20  & 463.77  & 250.95 & 269.48 & 0.166 & 0.176 \\
200 & 284.76  & 336.67  & 172.98 & 192.61 & 0.268 & 0.288 \\
300 & 201.98  & 240.25  & 106.35 & 149.36 & 0.361 & 0.379 \\
400 & 133.08  & 167.46  & 57.02  & 97.62  & 0.445 & 0.463 \\
500 & 67.85   & 94.08   & 23.31  & 53.98  & 0.519 & 0.533 \\
600 & 10.31   & 13.53   & 1.67   & 6.53   & 0.587 & 0.589 \\
619 & 0.00    & 0.00    & 0.00   & 0.00   & 0.600 & 0.600 \\
\bottomrule
\end{tabular}
\end{table}

\subsection{Operational and allocation outcomes}
\label{sec:supp_operational_distributional}

For each Pareto solution, expected evacuation demand, average transfer distance, hardening coverage, and average normalized hardening depth are
computed as
\begin{align}
\mathrm{EvacDemand}
&=
\sum_{s\in\mathcal S}P_s
\sum_{(j,k)\in\mathcal A_s}q_{sjk},
\\
\mathrm{AvgDistance}
&=
\frac{
\sum_{s\in\mathcal S}P_s
\sum_{(j,k)\in\mathcal A_s}d_{jk}q_{sjk}
}{
\sum_{s\in\mathcal S}P_s
\sum_{(j,k)\in\mathcal A_s}q_{sjk}
},
\\
\mathrm{HCov}
&=
\frac{
|\{j\in J^{\mathrm{send}}:y_j>0\}|
}{
|J^{\mathrm{send}}|
},
\\
\mathrm{HDepth}
&=
\frac{1}{|J^{\mathrm{send}}|}
\sum_{j\in J^{\mathrm{send}}}
\frac{y_j}{H_j^{\max}}.
\end{align}

For a facility subset $\mathcal J'\subseteq J^{\mathrm{send}}$, the share of
the hardening budget allocated to that group is
\[
\mathrm{BudgetShare}(\mathcal J')
=
\frac{
\sum_{j\in\mathcal J'}C_j^Hy_j
}{
\sum_{j\in J^{\mathrm{send}}}C_j^Hy_j
}.
\]

Table~\ref{tbl:supp_operational_metrics} shows that larger budgets reduce
both expected and extreme evacuation demand while expanding hardening
coverage and depth. 
\begin{table}[H]
\centering
\small
\caption{Operational outcomes for representative Pareto policies. Hardening
metrics are reported as percentages of exposed facilities and maximum
feasible protection levels.}
\label{tbl:supp_operational_metrics}
\resizebox{\textwidth}{!}{
\begin{tabular}{llrrrrrrrr}
\toprule
Budget & Policy
& \multicolumn{3}{c}{Evacuation demand}
& \multicolumn{3}{c}{Distance (miles)}
& Coverage
& Avg.\ depth \\
\cmidrule(lr){3-5}
\cmidrule(lr){6-8}
& & Expected & Median & Maximum
& Expected & Median & Maximum
& (\%) & (\%) \\
\midrule

\multirow{3}{*}{\$50M}
& Impact min. & 344.48 & 0 & 6720 & 3.68 & 0    & 6.31 & 10.66 & 4.40 \\
& Knee        & 342.37 & 0 & 6834 & 3.44 & 0    & 4.09 &  9.02 & 4.00 \\
& Cost min.   & 350.85 & 8 & 6834 & 3.36 & 0.42 & 4.09 &  9.84 & 3.87 \\

\midrule

\multirow{3}{*}{\$300M}
& Impact min. & 140.78 & 0 & 2832 & 2.29 & 0 & 2.49 & 42.62 & 37.79 \\
& Knee        & 117.78 & 0 & 2781 & 2.24 & 0 & 2.39 & 43.44 & 34.61 \\
& Cost min.   &  99.81 & 0 & 2275 & 3.13 & 0 & 3.35 & 48.36 & 40.22 \\

\midrule

\multirow{3}{*}{\$600M}
& Impact min. & 12.63 & 0 & 197 & 1.88 & 0 & 2.15 & 95.90 & 93.76 \\
& Knee        &  5.19 & 0 & 140 & 1.52 & 0 & 1.52 & 93.44 & 90.79 \\
& Cost min.   &  1.93 & 0 &  52 & 1.53 & 0 & 1.53 & 95.90 & 95.20 \\

\bottomrule
\end{tabular}
}
\end{table}

Table \ref{tbl:supp_facility_type_investment} shows hardening investment patterns by facility type at representative
solutions. 

\begin{table}[!t]
\centering
\small
\setlength{\tabcolsep}{4pt}
\caption{Hardening investment patterns by facility type at representative
solutions. The exposed set contains 23 hospitals and 99 nursing homes.}
\label{tbl:supp_facility_type_investment}
\begin{tabular}{lllrrr}
\toprule
Budget & Policy & Type & Coverage (\%) & Avg.\ depth (\%) & Budget share (\%) \\
\midrule

\multirow{6}{*}{\$50M}
& Impact min. & Hospital     & 17.39 & 6.58 & 41.12 \\
&             & Nursing home & 9.09 & 3.89 & 58.88 \\
\cmidrule(lr){2-6}
& Knee        & Hospital     & 17.39 & 9.60 & 54.81 \\
&             & Nursing home & 7.07 & 2.70 & 45.19 \\
\cmidrule(lr){2-6}
& Cost min.   & Hospital     & 21.74 & 9.58 & 58.61 \\
&             & Nursing home & 7.07 & 2.55 & 41.39 \\

\midrule

\multirow{6}{*}{\$300M}
& Impact min. & Hospital     & 47.83 & 44.64 & 31.51 \\
&             & Nursing home & 41.41 & 36.20 & 68.49 \\
\cmidrule(lr){2-6}
& Knee        & Hospital     & 78.26 & 75.65 & 41.85 \\
&             & Nursing home & 35.35 & 25.08 & 58.15 \\
\cmidrule(lr){2-6}
& Cost min.   & Hospital     & 100.00 & 100.00 & 45.34 \\
&             & Nursing home & 36.36 & 26.33 & 54.66 \\

\midrule

\multirow{6}{*}{\$600M}
& Impact min. & Hospital     & 95.65 & 94.78 & 22.09 \\
&             & Nursing home & 95.96 & 93.52 & 77.91 \\
\cmidrule(lr){2-6}
& Knee        & Hospital     & 100.00 & 100.00 & 22.67 \\
&             & Nursing home & 91.92 & 88.65 & 77.33 \\
\cmidrule(lr){2-6}
& Cost min.   & Hospital     & 100.00 & 100.00 & 22.67 \\
&             & Nursing home & 94.95 & 94.08 & 77.33 \\

\bottomrule
\end{tabular}
\end{table}

Figure~\ref{fig:supp_investment_patterns} reports hardening coverage and depth
by vulnerability group across the full Pareto frontiers. The strongest reallocation occurs at the \$300M budget, where tightening the
service-disruption constraint increases protection of facilities serving
higher-vulnerability communities. At \$600M, differences across vulnerability
groups narrow as most exposed facilities approach full protection.

\begin{figure}[!t]
\centering
\setlength{\tabcolsep}{4pt}
\begin{tabular}{cccc}
& \textbf{\$50M} & \textbf{\$300M} & \textbf{\$600M} \\

\rotatebox{90}{\begin{tabular}{c}
Hardening\\coverage (\%)
\end{tabular}}
&
\includegraphics[width=0.30\textwidth]{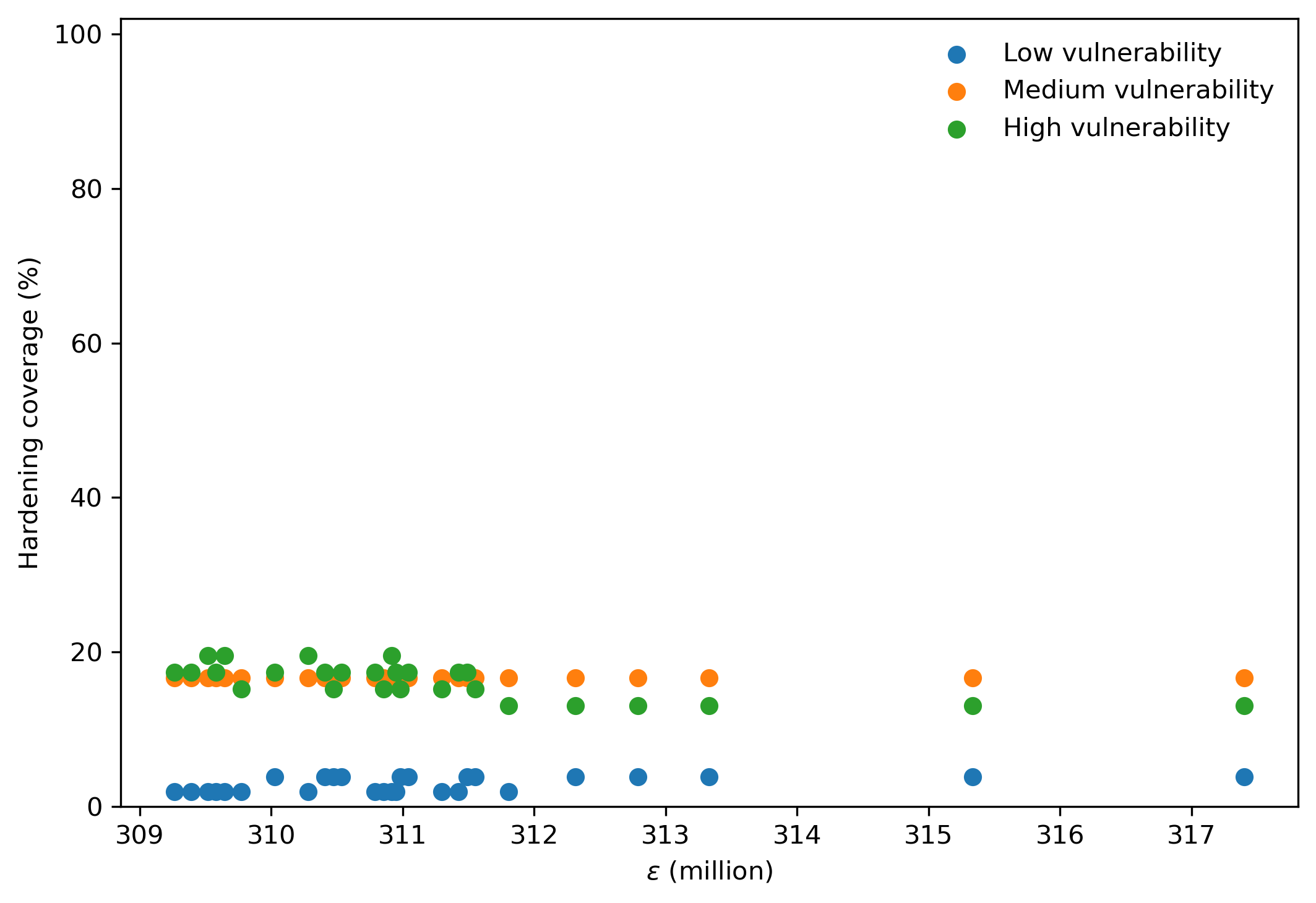}
&
\includegraphics[width=0.30\textwidth]{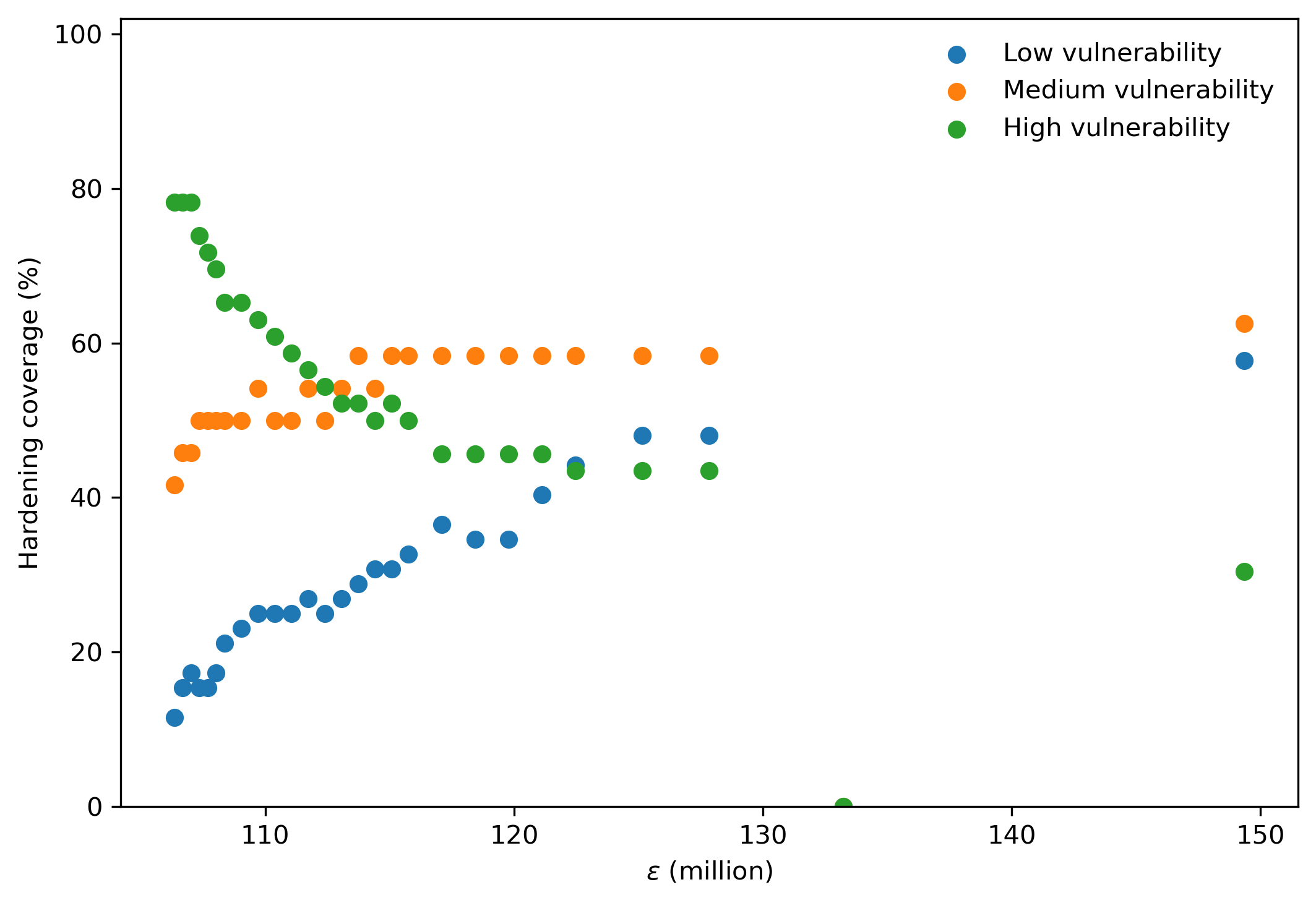}
&
\includegraphics[width=0.30\textwidth]{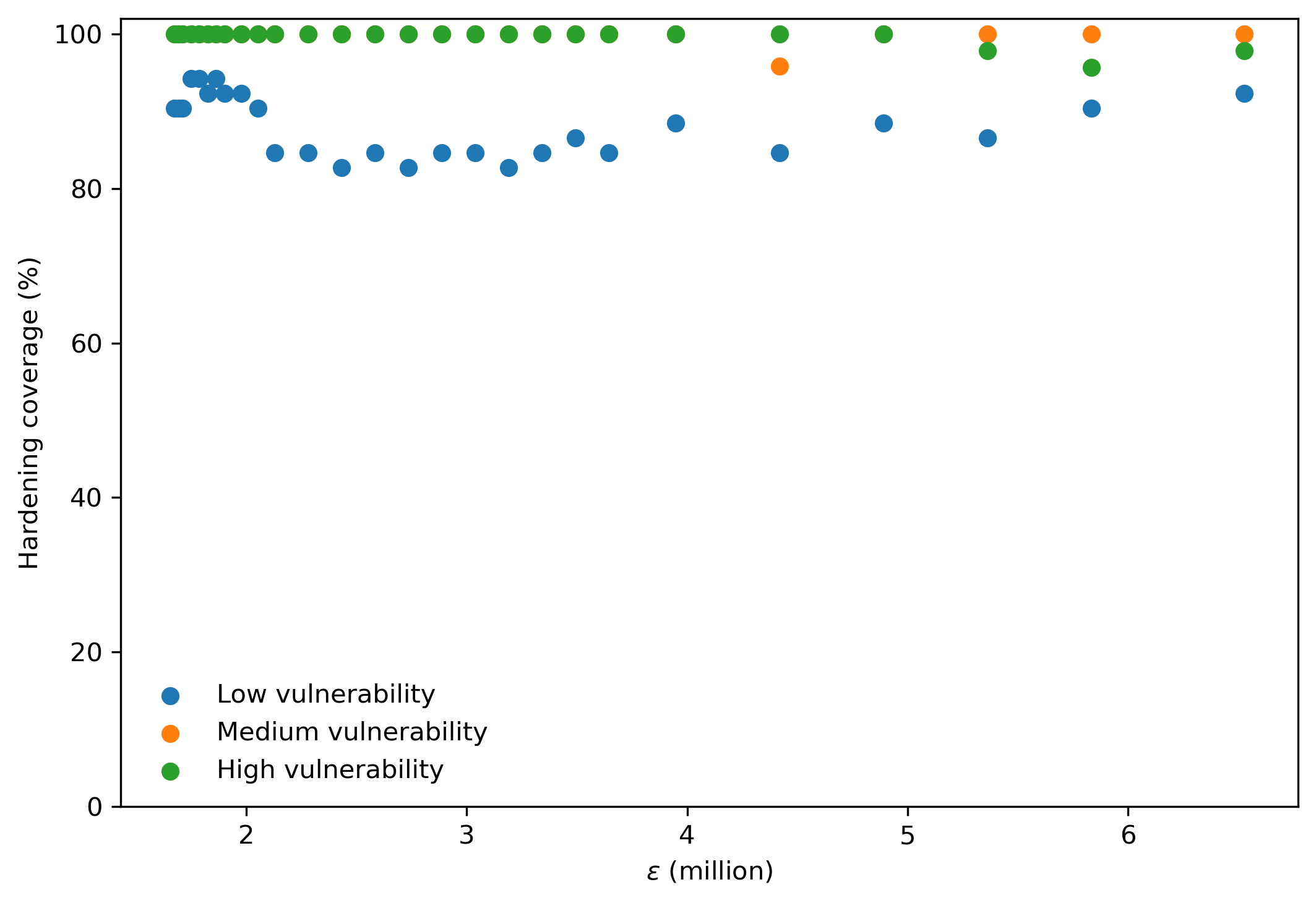}
\\[5pt]

\rotatebox{90}{\begin{tabular}{c}
Average normalized\\hardening depth (\%)
\end{tabular}}
&
\includegraphics[width=0.30\textwidth]{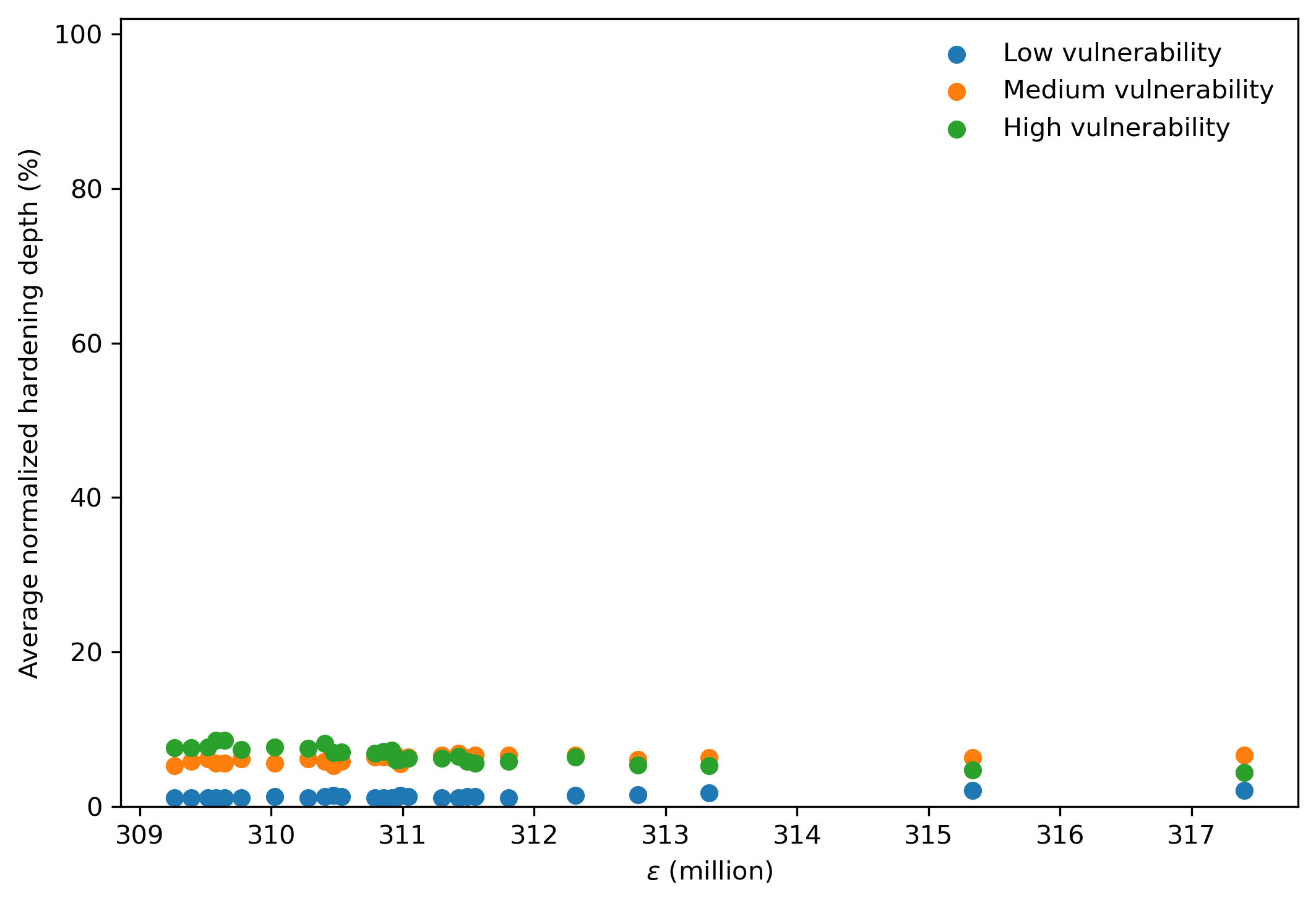}
&
\includegraphics[width=0.30\textwidth]{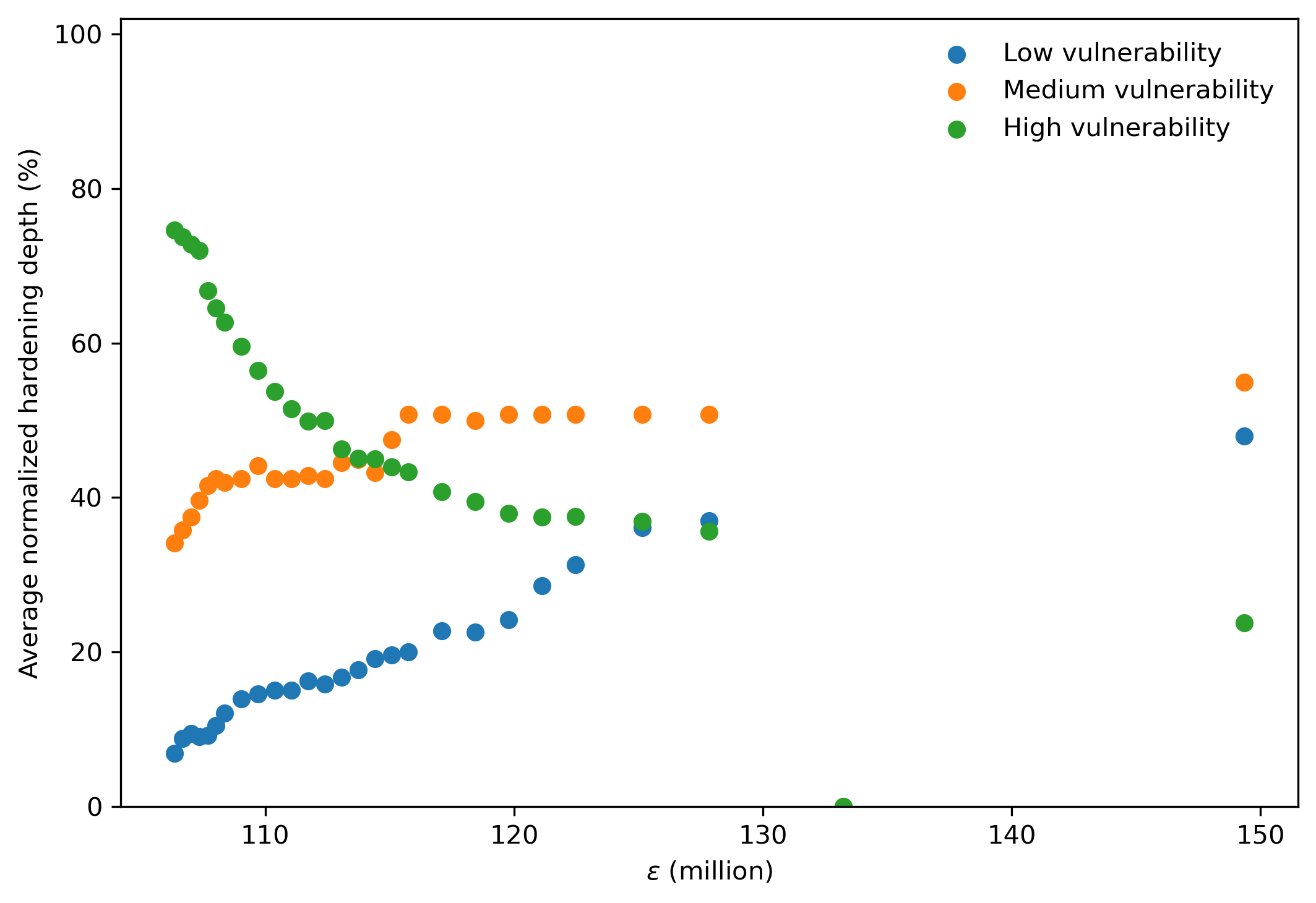}
&
\includegraphics[width=0.30\textwidth]{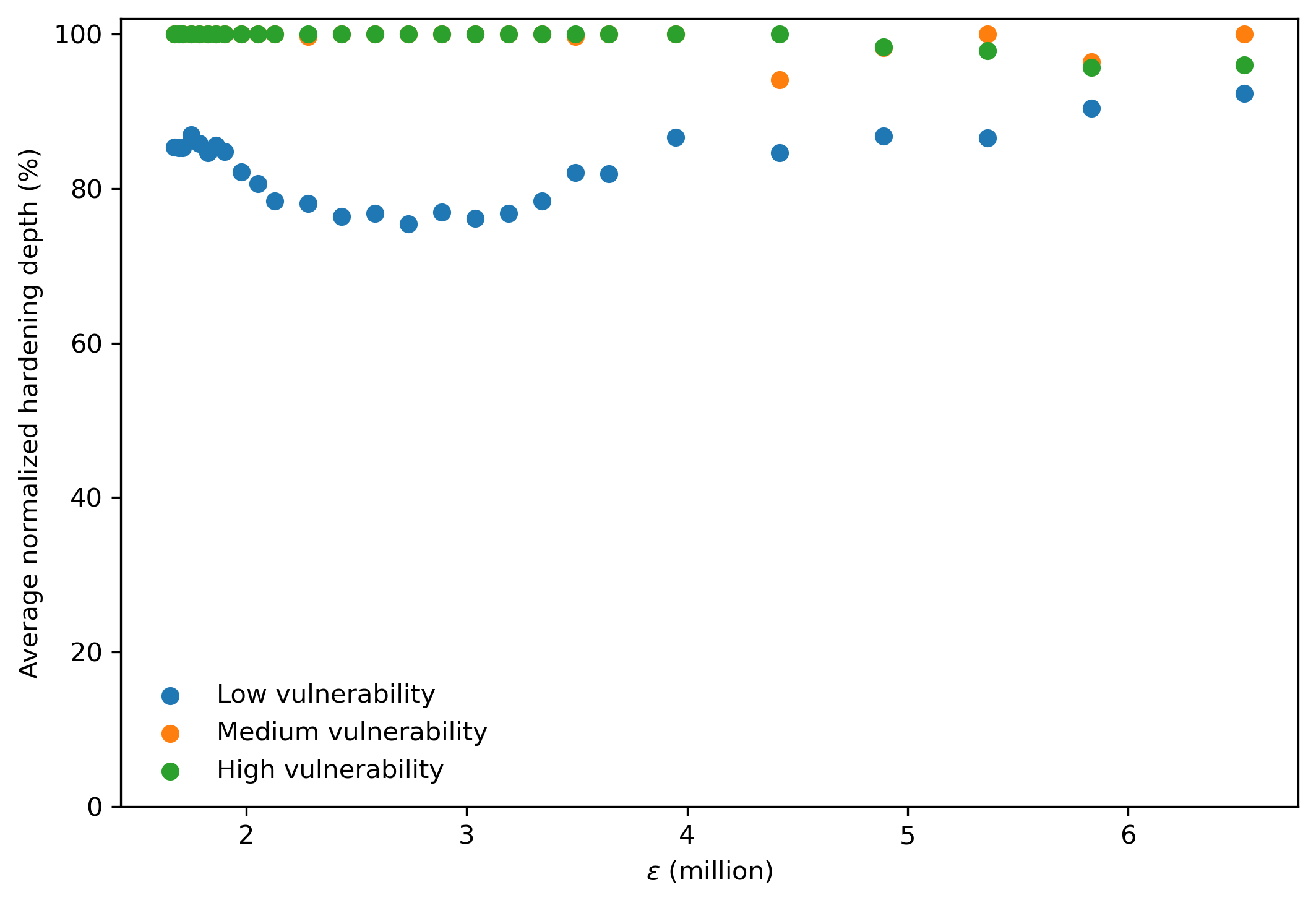}
\end{tabular}

\caption{Hardening coverage and normalized protection depth for facilities
serving low-, medium-, and high-vulnerability communities along the Pareto
frontiers.}
\label{fig:supp_investment_patterns}
\end{figure}

\FloatBarrier
\subsection{Detailed computational performance}
\label{sec:supp_algorithmic_performance}

Table~\ref{tbl:algorithm_performance} summarizes total runtime and per-point statistics over the 25 interior $\varepsilon$-points.

\begin{table}[!t]
\centering
\small
\setlength{\tabcolsep}{4pt}
\renewcommand{\arraystretch}{1.1}
\caption{Algorithmic performance for Pareto frontier generation}
\label{tbl:algorithm_performance}
\begin{tabular}{cc|c|ccc}
\toprule
\textbf{Budget} 
& \textbf{Method}
& \begin{tabular}{c}
Total Time \\
(min)
\end{tabular}
& \begin{tabular}{c}
Mean / pt \\
(sec)
\end{tabular}
& \begin{tabular}{c}
Median / pt \\
(sec)
\end{tabular}
& \begin{tabular}{c}
Max / pt \\
(sec)
\end{tabular} \\
\midrule

\multicolumn{6}{c}{\hspace{3em} $\mathcal{S}^{\text{main}}$ ($|\mathcal{S}|=27$)} \\
\midrule
\$50M  & Extensive & 235 & 564 & 563 & 614 \\
       & Benders     &  26  &  63  &  54  &  205  \\
       & Lagrangian  &  11  &  27  &  24  &  38  \\
\cmidrule(lr){1-6}
\$300M & Extensive & 288 (OOM at 2/25) & 752 & 692 & 1{,}714 \\
       & Benders     &  55  &  132  &  94  &  527  \\
       & Lagrangian  &  12  &  28  &  25  &  58  \\
\cmidrule(lr){1-6}
\$600M & Extensive   &  207 (OOM at 2/25)  &  540  &  495  &  993  \\
       & Benders     &  28  &  66  &  46  &  356  \\
       & Lagrangian  &  15  &  35  &  33  &  43  \\

\midrule
\multicolumn{6}{c}{\hspace{3em} $\mathcal{S}^{\text{perc}}$ ($|\mathcal{S}|=81$)} \\
\midrule
\$50M  & Extensive   &  -- (OOM at 25/25)  &  --  &  --  &  --  \\
       & Benders     &  83  &  200  &  159  &  822  \\
       & Lagrangian  &  28  &  67  &  67  &  90  \\
\cmidrule(lr){1-6}
\$300M & Extensive  &  -- (OOM at 25/25) &  --  &  --  &  --  \\
       & Benders     &  142  &  340  &  316  &  871  \\
       & Lagrangian  &  25  &  61  &  57  &  77  \\
\cmidrule(lr){1-6}
\$600M & Extensive   &  -- (OOM at 25/25) &  --  &  --  &  --  \\
       & Benders     &  107  &  257  &  207  &  1023  \\
       & Lagrangian  &  33  &  80  &  83  &  96  \\

\bottomrule
\end{tabular}

\vspace{2mm}

\begin{minipage}{\textwidth}
\footnotesize
Notes: Reported times include the 25 interior $\varepsilon$-points and exclude the 
lexicographic endpoint solves. Total time includes both model-building and solution times.
OOM denotes \emph{Out of Memory}, indicating solver termination due to memory exhaustion.
When OOM occurs, reported statistics are computed over successfully completed
$\varepsilon$-solves only.
\end{minipage}
\end{table}

\FloatBarrier